\documentclass[12pt,a4paper]{article}
\usepackage{bbm}
\usepackage{graphicx, color, tikz}
\graphicspath{ {./images/} }
\usepackage{wrapfig}
\usepackage{multicol}
\usepackage[mathscr]{euscript}
\usepackage{subcaption}
\usepackage{hyperref}
\usepackage{epstopdf}
\usepackage{enumerate}
\usepackage{amsmath,amsfonts,amssymb,amsthm,epsfig,epstopdf,titling,url,array}
\usepackage{color}
\usepackage{framed}
\usepackage[utf8]{inputenc}
\usepackage[english]{babel}

\usepackage{xparse}
\usepackage[margin=2.5cm]{geometry}
\usepackage[most]{tcolorbox}
\usepackage[export]{adjustbox}
\usetikzlibrary{decorations.markings}
\NewDocumentCommand{\INTERVALINNARDS}{ m m }{
	#1 {,} #2
}
\NewDocumentCommand{\interval}{ s m >{\SplitArgument{1}{,}}m m o }
{
	\IfBooleanTF{#1}{
		\left#2 \INTERVALINNARDS #3 \right#4
	}{
		\IfValueTF{#5}{
			#5{#2} \INTERVALINNARDS #3 #5{#4}
		}{
			#2 \INTERVALINNARDS #3 #4
		}
	}
}
\newtheorem{theorem}{Theorem}[section]
\newtheorem{corollary}{Corollary}[theorem]
\newtheorem{definition}{Definition}[section]
\newtheorem{lemma}[theorem]{Lemma}
\newtheorem{example}{Example}[section]

\newtheorem{Remark}{Remark}[section]

\newtheorem{question}[theorem]{Question}

\usepackage{authblk}

\begin{document} 
	\title{\textbf{Exponential polynomials with Baker omitted value}}
	\author{Sukanta Das\footnote{Corresponding author, a21ma09005@iitbbs.ac.in}}
	\author[2]{Subhasis Ghora\footnote{subhasisghora@iiitdm.ac.in}}
	\author[3]{Tarakanta Nayak\footnote{ tnayak@iitbbs.ac.in}}
	
	%	\footnote{
		%			Supported by University Grants Commission, Govt. of India.}
	\affil[1,3]{Department of Mathematics, School of Basic Sciences, 
		Indian Institute of Technology Bhubaneswar, India}
	\affil[2]{Department of Science and Humanities, Indian Institute of Information Technology, Design and Manufacturing, Kancheepuram, India}
	\date{}
	\maketitle
	\begin{abstract}  We consider exponential polynomials of the form
		$F_l(z)=P_1(z)\exp(Q_1(z))+P_2(z)\exp(Q_2(z))+\cdots+P_{l-1}(z)\exp(Q_{l-1}(z))+P_l(z)$,
		where $l\geq2$, each $Q_i$ is a non-constant polynomial and each $P_i$, $i=1,2,3\ldots,l-1$, is a polynomial (possibly constant), while $P_l$ is a non-constant polynomial. This article studies the topology of the preimages under $F_l$ of the neighborhoods of the essential singularity at infinity. We prove that for each disk $D$ (with respect to the spherical metric) centered at infinity, the set $F_l ^{-1} (D) $ is connected and, in fact, is an infinitely connected domain with all its boundary components bounded. In such a situation, the point at $\infty$ is called the Baker omitted value of the function.  Our methods extend those developed by Das and Nayak for the function $\exp(z)+P(z)$ for any non-constant polynomial $P$ (Complex Var. Elliptic Equ. 70 (2025), 1831–1847), providing a broader framework for studying this phenomenon. We further show that these functions do not admit any Baker wandering domain. Finally, we conclude by posing some problems arising out of this work.
		
	\end{abstract}
	\textit{Keywords:} Exponential polynomial; Baker omitted value; Baker wandering domain.\\
	
	AMS Subject Classification: 37F10, 30D05
	%====================================================================	
	\section{Introduction}
	Let $f:\mathbb{C}\to\mathbb{C}$ be a transcendental entire function. A point $a\in\widehat{\mathbb C}$ is called a singular value of  $f$  if at least one branch of its inverse $f^{-1}$ fails to be defined at $a$. A complex number $z_0$ is called a critical point of $f$ if $f'(z_0)=0$, and $f(z_0)$ is called a critical value of $f$. Clearly, a critical value is a singular value. To describe other types of singular values, let $D_r(a)$ be a disk (in the spherical metric) for some $r>0$ and choose a component $U_r$ of $f^{-1}(D_{r}(a))$ such that $U_{r_{1}}\subset U_{r_{2}}$ for $0 < r_1 < r_2$. If $\bigcap_{r>0}U_{r}=\emptyset$, then the choice $r\mapsto U_r$ is said to define a transcendental singularity of $f^{-1}$, and we say that a transcendental singularity lies over $a$. In this situation, there exists a curve $\Gamma:[0,\infty)\to \mathbb{C}$ such that $\lim_{ t\to \infty}\Gamma(t)=\infty$ and $\lim_{ t\to \infty}f(\Gamma(t))=a$, and $a$ is called an asymptotic value of $f$.  
	A singularity $U_r$ lying over $a$ is called {\it{logarithmic}} if $f:U_r\to {D_r(a)\setminus\{a\}}$ is a universal covering for some $r>0$. In particular, $U_r$ is simply connected. Otherwise, the singularity is called non-logarithmic.  For instance, the function $\exp(z)$ has a logarithmic singularity lying over $0$, which is a primary reason for the extensive study of its dynamics - its Fatou and the Julia sets (for details, one can see \cite{Berg 1993}).  However, functions with non-logarithmic singularities can be very wild as long as their dynamics are concerned. This article deals with functions having non-logarithmic singularities of a special type. Some definitions are needed to make this more precise.
	\par A point $a\in\mathbb{\widehat{C}}$ is said to be an omitted value of $f$ if $f(z)\neq a$ for any $z\in\mathbb{C}$. In particular, the point at infinity is always an omitted value for every transcendental entire function. Moreover, an omitted value is always an asymptotic value of the function \cite{Iversen 1914}. In 2016, Chakra et al. introduced {\it{Baker omitted value}} (in short, {\it{bov}}), over which there is a single singularity, and that is not logarithmic  \cite{ChakraChakrabortyNayak 2016}. 
	\begin{definition}
		An omitted value $a\in\widehat{\mathbb{C}}$ of an entire function $f$ is said to be bov if there is a disk $D$ with center at $a$ such that $f^{-1}(D)$ is infinitely connected and each component of the boundary of $f^{-1}(D)$ is bounded.
	\end{definition} 
	In fact, if $a$ is a bov of $f$ then for any $r>0$, $f^{-1}(D_r)$ is infinitely connected, and each of its boundary components is bounded where $D_r$ denotes a disk with center at $a$ and radius $r$ (Lemma 2.3, \cite{ChakraChakrabortyNayak 2016}). Moreover, functions with bov have an unbounded set of singular values (Lemma 2.5, \cite{GhoraNayakSahoo 2021}) and therefore do not belong to the well-studied Eremenko-Lyubich class  \cite{GhoraNayakSahoo 2021}. This is a motivation to take up the investigation of these functions from a dynamical perspective. 
	\par 
	Now, we turn to the dynamical relevance of bov.	The first example of an entire function having a Baker wandering domain - a multiply connected bounded domain that wanders and at the same time expands in all directions towards infinity under the iteration of the function (see Section \ref{Sec2} for a precise definition) is given by Baker in 1976 \cite{Baker 1976}. Interestingly, Baker had already introduced and studied this function in 1963, but at that time he was unable to establish that it has wandering domains \cite{Baker 1963}. The function was 	\begin{equation}\label{Exm1}
		f(z)=\frac{z^2}{4e}\prod_{n=1}^{\infty}\left(1+\frac{z}{r_n}\right),
	\end{equation}
	where  $r_{n+1}=\frac{r_n^2}{4e}\prod_{i=1}^{n}\left(1+\frac{r_n}{r_i}\right)$, $n\in\mathbb{N}$ and $r_1>4e$. It is observed (in Theorem 2.3, \cite{ChakraChakrabortyNayak 2016}) that this function has bov and that this is a consequence of Baker wandering domain. The notion of Baker omitted value, in fact, traces its origin to this important fact  \cite{ChakraChakrabortyNayak 2016}. Some other examples of functions with Baker wandering domains are inspired by (\ref{Exm1}) and can be found in \cite{Baker 1985, Baker Dominguez 2000, BergZheng 2011, BergRippStall 2013}. All the entire functions studied in these articles are infinite products of polynomials. On the other hand, entire functions in \textit{closed form} admitting a bov have also been found. For example, it is proved that  $\exp(z) +cz^d$ has bov for every non-zero complex number $c$ and every natural number $d$ (see  Remark 2.2,   \cite{GhoraNayakSahoo 2021}). In 2025, Das and Nayak \cite{DasNayak 2024} proved that for every non-constant polynomial $P$ and natural number $k$, the function $E^k(z)+P(z)$ has bov where $E^k(z)$ is the $k$-times composition of $E(z)=\exp(z)$ with itself. The existence of bov for these closed form functions is proved not as a consequence of Baker wandering domains (which is usually very much involved itself) but by a completely function-theoretic argument. This is a departure from Baker's example.
	
	\par Note that for $k=1$ the function  $E^k(z)+P(z)$ is of order $1$ and  is of infinite order when $k>1$. Here, the order of an entire function $f$ is given by $\rho(f):=\limsup_{r\to\infty}\frac{\log\log M(r,f)}{\log r}$, where $M(r,f)=\max_{|z|=r}|f(z)|$ (see \cite{asymptoticvalue-book} for more details).   This article identifies entire functions of finite order, especially those of integral order greater than one that admit bov. We consider exponential polynomials, i.e., functions of the form \begin{equation}\label{EXP-poly}
		F_l (z):=	P_1(z)\exp(Q_1(z))+P_2(z)\exp(Q_2(z))+\cdots+P_{l-1}(z)\exp(Q_{l-1}(z))+P_l(z), 
	\end{equation}
	where $l \geq 2$,  $Q_i$ is a non-constant polynomial and $P_i$ is a non-zero polynomial (possibly constant, i.e., $P_i \equiv c, c\neq 0 $) for each $i=1,2,\ldots, l-1$ and $P_l$ is a non-constant polynomial. From now on, a constant polynomial $P$ taking the value $c$ is denoted by $P \equiv c$. If a $Q_i$ is constant, then the corresponding term $P_i \exp(Q_i)$ becomes a polynomial and gets absorbed in $P_l$, leading to a function of the same form as (\ref{EXP-poly}). Similarly, unless $P_l$ is non-constant,  the function $F_l$ might have a finite asymptotic value (e.g., $e^z+1$) ruling out the existence of bov  (see Theorem 2.1,~\cite{ChakraChakrabortyNayak 2016}) - a situation we do not consider. The order of $F_l$ (with all the aforesaid restrictions on $P_i$ and $Q_i$) is  $\max \{\deg(Q_i): 1 \leq i \leq l-1\}$. We prove that for each natural number $d$, there is an entire function of order $d$ having bov, namely $P_1 \exp(Q_1)+P_2$, where the degree $\deg(Q_1)=d$. Thus, our first result concerns $F_2=P_1 \exp(Q_1)+P_2$ which generalizes Theorem 1.1, \cite{DasNayak 2024} (where $P_1(z)=1$ and $Q_1(z)=z$). Moreover, we show that $F_l$ has bov for $l\geq3$ under certain additional hypotheses, which will be stated just after the following theorem.

	\begin{theorem}\label{P_1e^{Q_1}+P_2}
		If  $Q_1$  and $P_2$ are non-constant polynomials, and $P_1$ is a non-zero polynomial, which is possibly constant, then the exponential polynomial $F_2= P_1\exp(Q_1)+P_2$ has bov.
	\end{theorem}
	
	\par To prove the existence of bov for $F_3$, we need a hypothesis on fundamental rays, which we now define.
	For each non-zero complex number $z_0$, it is possible to choose a branch of the argument such that it is continuous at $z_0$. For example, one may choose the principal argument if $z_0$ is not a negative real number and the branch $0\leq arg(z)< 2\pi$ otherwise. Denote the argument by $Arg$, and throughout this article, $Arg$ denotes argument modulo $2 \pi$.  For a polynomial $Q$ of degree $m$, let $\beta_m$ denote the argument of its leading coefficient. Define $\alpha_0(Q)=-\frac{\beta_m}{m}-\frac{\pi}{2m}$ and $\alpha_j(Q)=\alpha_0(Q)+\frac{\pi j}{m}$ for $j=1,2,\ldots,2m-1$. For each $j=0,1,2,\ldots,2m-1$, the ray $R_{\alpha_j(Q)}:=\{z:\operatorname{Arg}(z)=\alpha_j(Q)\}$ is called a \textit{fundamental ray} of $Q$; more precisely, it is referred to as the $j$-th fundamental ray of $Q$. By definition, any two fundamental rays of a polynomial $Q$ are disjoint. However, fundamental rays of two different polynomials $Q_1$ and $Q_2$ either coincide or are disjoint. Section \ref{Growth-exponential-polynomial} contains more details on fundamental rays.  
	\begin{theorem}\label{e^{Q_1 }+e^{Q_2 }+P_3}
		Let   $Q_1, Q_2$ and $P_3$ be non-constant polynomials, and $P_1$, $P_2$ are non-zero polynomials which are possibly constants. If $\deg(Q_1)\neq \deg(Q_2)$ and none of the fundamental rays of $Q_1$ coincide  with those of $Q_2$, then the exponential polynomial $F_3 =P_1\exp(Q_1)+P_2\exp(Q_2)+P_3$ has bov.
	\end{theorem} 
	The assumption that $\deg(Q_1)\neq \deg(Q_2)$ in Theorem~\ref{e^{Q_1 }+e^{Q_2 }+P_3} is just a sufficient condition. If $\deg(Q_1)=\deg(Q_2)$ then $P_1e^{Q_1 }+P_2e^{Q _2 }+P_3$ may reduce to a polynomial, and Theorem~\ref{e^{Q_1 }+e^{Q_2 }+P_3} may fail. Taking $P_1=P_2 \equiv 1$ and   $Q_1-Q_2\equiv k\pi i$ for an odd integer $k$, it is seen that $e^{Q_1}+e^{Q_2}+P_3$ becomes $P_3$, a polynomial. Similarly, if $P_1\equiv 1$, $P_2\equiv  -1$ and $Q_1-Q_2\equiv k\pi i$ for an even integer $k$ then $e^{Q_1}-e^{Q_2}+P_3$ reduces to $P_3$. Moreover, Section \ref{Examples} provides examples of exponential polynomials of the form $F_3$ for which $\deg(Q_1)=\deg(Q_2)$. It is also interesting that $\deg(Q_1)\neq \deg(Q_2)$ is not enough to ensure that the fundamental rays of $Q_1$ and $Q_2$ do not coincide. The following remark makes it precise. 
	\begin{Remark}
		Let $\beta_i$ be the argument of the leading coefficient of $Q_i$ for $i=1,2$. Then the $j$-th fundamental ray of $Q_1$ coincides with the $k$-th fundamental ray of $Q_2$  if $\frac{-\beta_1}{m_1}-\frac{\pi}{2 m_1}+\frac{\pi j}{m_1} \equiv \frac{-\beta_2}{m_2}-\frac{\pi}{2 m_2}+\frac{\pi k}{m_2}$ mod ($2 \pi$) where $j \in \{ 0,1,2,\ldots, 2m_1 -1\}$ and $k \in \{0,1,2,\ldots ,2m_2 -1\}$. This is equivalent to the statement that $\frac{m_1 (\beta_2 +\frac{\pi}{2}) -m_2 (\beta_1 +\frac{\pi}{2})}{ \pi \gcd(m_1, m_2) }$ is an integer.
	\end{Remark} 

The next theorem extends the preceding results to exponential polynomials with an arbitrary number of terms. 
	\begin{theorem}\label{Generalization}
		For a natural number $l \geq 2,$
		let $Q_i$  be a non-constant polynomial, and  $P_i$ be a non-zero polynomial which is possibly  constant for $ i=1,2,\ldots,l-1$. Also, let $P_l$ be a non-constant polynomial.  If $\deg(Q_i)\neq \deg(Q_j)$  and none of the  fundamental rays of $Q_i$ coincide with those of $Q_j$ for any $i\neq j$, then the exponential polynomial $F_l = P_1 \exp(Q_1) +P_2\exp(Q_2)+\cdots+P_{l-1}\exp(Q_{l-1})+P_l$ has bov.
	\end{theorem}
	%%%%%%%%%%%%%%%%%%%%%%%%%%%
	Beginning with the classical works of Pólya \cite{Polya}, exponential polynomials have been extensively studied from the viewpoint of value distribution (see \cite{Steinmetz} and more recent investigations by Heittokangas, Ishizaki, Tohge and Wen \cite{Heittokangas et al. 2021, Heittokangas Wen, Heittokangas et al. 2023}). These studies mainly concern asymptotic growth, the distribution of zeros, the value distribution of exponential polynomials, and their roles in the theories of complex differential equations and oscillation theory. In contrast, the present work considers exponential polynomials from a dynamical perspective, thereby complementing our understanding of these objects.
	
	\par The key ingredient of our proofs is a careful analysis of the behavior of the exponential of a polynomial in the fundamental sectors - sectors bounded by the fundamental rays of the polynomial. 
	
	%%%%%%%%%%%%%%%%%%%%%%%%%%%
		\par The structure of the paper is as follows. Sections \ref{Sec2}  and \ref{Growth-exponential-polynomial}  discuss useful results concerning the relation of bov and Baker wandering domains, argument of polynomials along unbounded curves and growth of the exponential of a polynomial respectively. The proofs of Theorems \ref{P_1e^{Q_1}+P_2}, \ref{e^{Q_1 }+e^{Q_2 }+P_3} and \ref{Generalization} are given in Section \ref{Proofs}. In Section \ref{No BWD}, it is shown in Theorem~\ref{No-BWD-result} that $F_l$ without the restriction on the coincidence of fundamental rays of $Q_i$  cannot have any Baker wandering domain. Examples of exponential polynomials having bov but not satisfying the hypothesis of Theorem~\ref{e^{Q_1 }+e^{Q_2 }+P_3} are given in Section \ref{Examples}. The article concludes with a set of problems proposed for further investigation.  
	
	\section{ Baker omitted value and Baker wandering domains}\label{Sec2}
	After discussing some relevant facts on bov, we present some known useful lemmas in this section. 
	
	\par In 2016, Chakra et al. \cite{ChakraChakrabortyNayak 2016} proved that if there is a bov of a transcendental entire function $f$, then it must be $\infty$. Indeed, if $a \in\mathbb{C}$ is an omitted value of $f$  and $D$ is a disk centered at $a$, then each component of the set $f^{-1}(D)$ is unbounded (as $a$ is omitted) and simply connected (a consequence of the Maximum Modulus Principle). Thus, the boundary of each component of $f^{-1}(D)$ is unbounded, so $a$ is not a bov.
	\par    
	By saying an entire function has bov, we mean that it has bov at $\infty$. Also, it can be seen that the bov is the only asymptotic value of the function (see Theorem 2.1,~\cite{ChakraChakrabortyNayak 2016}). Later, in 2023, Ghora et al. proved that the set of all singular values of every entire function with bov is unbounded, by showing that the bov is always a limit point of critical values (see Lemma 2.5, \cite{GhoraNayakSahoo 2021}). Chakra et al. \cite{ChakraChakrabortyNayak 2016} established a necessary and sufficient condition for the existence of bov, which serves as a fundamental tool in the proofs of our main results. A curve $\gamma:[0,\infty)\to\mathbb{C}$ is said to be unbounded if for any $M>0$, there exists some $t_0\in[0,\infty)$ such that $|\gamma(t_0)|>M$. Throughout the article a curve is understood to be continuous unless stated otherwise.
	\begin{lemma}\label{Bov iff}
		An entire function $f$ has bov if and only if $f(\gamma)$ is unbounded for each unbounded (connected) curve $\gamma$.	
	\end{lemma}
	Here is a straightforward but useful remark.
	\begin{Remark}\label{unboundedpart-unboundedcurve}
		If an entire function $f$ has bov and $\gamma$ is an unbounded curve, then for every $M>0$, the set $\{f (z): z \in \gamma ~\mbox{and}~ |z|>M\}$ is unbounded.
	\end{Remark}
	
	It follows from Lemma~\ref{Bov iff} that no other finite asymptotic value can exist for a function in the presence of bov. Further, it is well-known that if a finite point has finitely many preimages under an entire function, then it must be an asymptotic value of the function. This leads to the following remark.
	
	\begin{Remark}\label{Rem2.1}
		If an entire function has bov, then every point, except infinity, has infinitely many preimages under $f$.
	\end{Remark}
	The iterative behaviour of an entire function is influenced by the bov. We first recall some standard notions. The set of points $z\in\mathbb{\widehat{C}}$ for which the sequence of iterates $\{f^n(z)\}_{n\geq 1}$ forms a normal family is called the Fatou set of $f$. The Julia set is the complement of the Fatou set in $\widehat{\mathbb{C}}$. A maximally connected subset of the Fatou set is called a Fatou component. For a Fatou component $U$ of $f$, let $U_n$ denote the Fatou component containing $f^n(U)$ for $n\in\mathbb{N}$.   A Fatou component $W$ is called a wandering domain if $W_n\bigcap W_m=\emptyset$ for all $m\neq n$. For two subsets $A,B$ of $\mathbb{C}$, we say $A$ surrounds  $B$ if there exists a bounded component of $\mathbb{{\widehat{C}}}\setminus A$   containing $B$.
	\begin{definition}
		A wandering domain $W$ is called a {\it{Baker wandering domain}} of an entire function $f$ if each $W_n$ is bounded, multiply connected, and there exists a natural number $N$ such that $W_{n+1}$ surrounds $W_n$ and $0$ for all $n\geq N$ and $f^n(z)\to \infty$ as $n\to \infty$ for all $z\in W$.
	\end{definition} 
	For more details on Baker wandering domain we refer the reader to the survey \cite{DasNayak 2026}. A necessary and sufficient condition for a Fatou component to be a Baker wandering domain is due to Baker (See Theorem 3.1 in \cite{Baker 1984}).
	\begin{lemma}\label{BWD-iff}
		A Fatou component of a transcendental entire function is a Baker wandering domain if and only if it is multiply connected. 
	\end{lemma}

	In 2016, Chakra et al. \cite{ChakraChakrabortyNayak 2016} established that the existence of a Baker wandering domain is sufficient for the existence of the bov.   
	\begin{lemma}\label{BWD to Bov}
		If a transcendental entire function $f$ has a Baker wandering domain, then $f$ has bov.
	\end{lemma}
	\par However, the converse of Lemma~\ref{BWD to Bov} does not hold in general. In fact, several examples of entire functions are known that admit bov but do not have any Baker wandering domain  (see \cite{ChakraChakrabortyNayak 2016, DasNayak 2024}). 
	
	An important remark is the following, and we shall use this in Section \ref{No BWD}.	
	\begin{Remark}\label{Rem2.2}
		From Remark \ref{Rem2.1} and Lemma \ref{BWD to Bov}, it follows that an entire function with finitely many zeros cannot have a Baker wandering domain.
	\end{Remark}

	\section{Growth of the exponential of a polynomial}\label{Growth-exponential-polynomial}
	In this section, we discuss some standard growth properties of the exponential of a polynomial and its asymptotic behavior near the fundamental rays of the polynomial.
	\par We first look at the behavior of a non-constant polynomial along  unbounded curves approaching rays  emanating from the origin. Let $\gamma:[0,\infty)\to\mathbb{C}$ be a curve satisfying $\gamma(t)\to\infty$ and $Arg(\gamma(t))\to\theta$ as $t\to\infty$, for some $\theta\in\mathbb{R}$. We study the limiting behavior of $Arg(P(\gamma(t)))$ as $t\to\infty$, where $P$ is a non-constant polynomial.
		\begin{lemma}\label{Arg(P) along ray}
		Let $P$ be a non-constant polynomial of degree $d$ with leading coefficient $a_d$. If $\gamma:[0,\infty)\to \mathbb{C}$ is a curve such that $\lim_{ t\to \infty}\gamma(t)=\infty$ and $\lim_{ t\to \infty}Arg(\gamma(t))=\theta$ for some $\theta \in \mathbb{R}$, then  $\lim_{ t\to \infty}Arg(P(\gamma(t)))=d\theta+Arg(a_d)$, where $Arg(.)$ is chosen such that it is continuous at $a_d$.
	\end{lemma}
	\begin{proof}
		Let $P(z)=a_0+a_1z+\cdots+a_{d-1}z^{d-1}+a_dz^d$,  where $a_d\neq0$. Since $\lim_{t \to \infty}\gamma(t)\to\infty$, we have $\gamma(t)\neq0$ for all sufficiently large $t$. For each such $t$, we have
		\begin{equation}\label{limiting}
			Arg(P(\gamma(t)))=dArg(\gamma(t))
			+Arg\left(\frac{P(\gamma(t))}{\gamma(t)^d}\right).	\end{equation}
		Moreover,
		$$\frac{P(\gamma(t))}{\gamma(t)^d}=a_d+\frac{a_{d-1}}{\gamma(t)}+\frac{a_{d-2}}{\gamma(t)^2}
		+\cdots+\frac{a_0}{\gamma(t)^d}.
		$$
		It follows that
		$\lim_{t\to\infty}\frac{P(\gamma(t))}{\gamma(t)^d}=a_d$.
		By the continuity of the argument in a sufficiently small neighbourhood of the non-zero number $a_d$, we obtain
		$$\lim_{t\to\infty}Arg\left(\frac{P(\gamma(t))}{\gamma(t)^d}\right)= Arg\left(\lim_{t\to\infty}\frac{P(\gamma(t))}{\gamma(t)^d}\right)=Arg(a_d).$$ Since $\lim_{t\to\infty}Arg(\gamma(t))=\theta$, it follows from Equation~(\ref{limiting}) that  
		$\lim_{t\to\infty}Arg(P(\gamma(t)))= d\theta+Arg(a_d).$  	
	\end{proof}
	The next corollary is the sequence version of the preceding lemma.
	\begin{corollary}\label{Cor1}
		If $P$ is a non-constant polynomial of degree $d$ at least two, the argument of whose leading coefficient is $\beta_d$ and  $\{z_n\}_{n\geq 1}$ is any sequence satisfying $\lim_{ n\to \infty} z_n=\infty$ such that for some $\theta\in\mathbb{R}$, $Arg(z_n)\to\theta$ as $n\to\infty$, then $Arg(P(z_n))\to d\theta+\beta_d$ as $n\to\infty$.
	\end{corollary}	
	\begin{proof}
		Let $\{z_n\}_{n\geq 1}$ be any sequence satisfying $\lim_{ n\to \infty} z_n=\infty$ such that $Arg(z_n)\to\theta$ as $n\to\infty$. Join   $z_n$ and $z_{n+1}$ by a line segment $L_n$ defined by $L_n (t)= (1-t+n)z_n +(t-n)z_{n+1},~ n \leq t \leq n+1$. Let  $\gamma:[1,\infty)\to \mathbb{C}$ be defined by $\gamma(t)=L_n(t),$ for $n \leq t \leq n+1$. Then, $\gamma$    
		is a continuous curve such that $\lim_{ t\to \infty}\gamma(t)=\infty$ and $\lim_{ t\to \infty}Arg(\gamma(t))=\theta$, then $ArgP(\gamma(t))   \to \beta_d+d\theta$ as $z\to\infty$ along $\gamma$. The rest follows from the Lemma \ref{Arg(P) along ray}. 
	\end{proof}
	\par In order to analyze the growth of exponential of a polynomial, we need to fix some notations. Let $Q$ be a polynomial with degree $m$ and $\beta_m$ denote the argument of its leading coefficient. Denote $\alpha_0(Q)=-\frac{\beta_m}{m}-\frac{\pi}{2m}$,  $\alpha_j(Q)= \alpha_0(Q) +\frac{\pi j}{m}$ for  $j=1,2,3,\ldots,2m-1$ and denote the sector $\{z\in\mathbb{C}: \alpha_0(Q)<Arg(z)<\alpha_1(Q)\}$ by $A_0(Q)$. Define \begin{equation}\label{A_j(P)}
		A_j(Q)=\{e^{\frac{i\pi}{m}}z: z\in A_{j-1}(Q)\} ~\mbox{for}~  j=1,2,3,\ldots,2m-1. 
	\end{equation}
	Note that $\{0\}\cup R_{\alpha_j(Q)}\cup R_{\alpha_{j+1}(Q)}$ constitutes the boundary of $A_j(Q)$. Moreover, the collection $\{A_0(Q),A_1(Q),A_2(Q),\ldots,A_{2m-1}(Q)\}$ together with $\bigcup_{j=0}^{2m-1}R_{\alpha_j(Q)} \cup \{0\}$ forms a partition of $\mathbb{C}$, and we call this  the {\it{fundamental partition}} (of the complex plane) with respect to $Q$. We  consider a closed sub-sector of $A_j(Q)$ given by
	\begin{equation}\label{closed subsector}
		\overline{B^\epsilon_j(Q)}=\left\{z:\alpha_j(Q)+\frac{\epsilon}{m}\leq Arg(z)\leq\alpha_{j+1}(Q)-\frac{\epsilon}{m}\right\},
	\end{equation} 
	where $0< \epsilon <\frac{\pi}{2m} \leq  \frac{\pi}{2}$ is chosen to be sufficiently small as per the requirement. For a cubic polynomial $Q$ with a positive leading coefficient, the six sectors and the corresponding closed sub-sectors are given in Figure~\ref{Fig2}.
	
	\par  Now for each $j$, we discuss the behavior of $\exp(Q)$ on these closed sub-sectors $\overline{B^\epsilon_j(Q)}$ and near the boundary of $A_j(Q)$. Though $0< \epsilon < \frac{\pi}{2}$ is enough to make sense of the closed sub-sector of $A_j (Q)$, we take $\epsilon$ in a smaller interval $(0, \frac{\pi}{4})$ for a purpose to be served in the proof of Theorem~\ref{P_1e^{Q_1}+P_2} (see  Lemma~\ref{Growth-exppoly}(1)).	To prove the theorems and explain the example, Lemma~\ref{Growth-exppoly}  will be used. Though this is already known (see p. 254, \cite{Markushevich}), we include it along with its proof for better clarification and completeness.
	
		\begin{figure}
		\begin{center}
			\begin{tikzpicture}[scale=0.9]
				\fill[white!30]
				(0,0) -- (3,1.732) -- (0,3) -- cycle;
				\fill[white!30]
				(0,0) -- (-3,1.732) -- (-3,-1.732) -- cycle;
				\fill[white!30]
				(0,0) -- (0,-3) -- (3,-1.732) -- cycle;
				\fill[white!30]
				(0,0) -- (3,1.732) -- (3,-1.732) -- cycle;
				\fill[white!30]
				(0,0) -- (0,3) -- (-3,1.732) -- cycle;
				\fill[white!30]
				(0,0) -- (0,-3) -- (-3,-1.732) -- cycle;
				\draw[<->,thick] (0,-4)--(0,4) node[above]{};
				\draw[black, ->,thick] (0,0)--(4,2.309) node[right]{};
				\draw[black, ->,thick] (0,0)--(4,-2.309) node[right]{};
				\draw[black, ->,thick] (0,0)--(4,2.309) node[right]{};
				\draw[black, ->,thick] (0,0)--(4,2.309) node[right]{};
				\draw[black, ->,thick] (0,0)--(-4,-2.309) node[right]{};
				\draw[black, ->,thick] (0,0)--(-4,2.309) node[right]{};
				\draw[black, ->,thick] (0,0)--(0,4) node[right]{};
				\draw[black, ->,thick] (0,0)--(0,-4) node[right]{};
				\draw[blue, ->, line width=0.7pt, dotted] (0,0)--(2.5,2.1324) node[right]{};
				\draw[blue, ->, line width=0.7pt, dotted] (0,0)--(-2.5,-2.1324) node[right]{};
				\draw[blue, ->, line width=0.7pt, dotted] (0,0)--(0.52,2.9714) node[right]{};
				\draw[blue, ->, line width=0.7pt, dotted] (0,0)--(-0.52,-2.9714) node[right]{};
				\draw[blue, ->, line width=0.7pt, dotted] (0,0)--(-0.52,2.9714) node[right]{};
				\draw[blue, ->, line width=0.7pt, dotted] (0,0)--(0.52,-2.9714) node[right]{};
				\draw[blue, ->, line width=0.7pt, dotted] (0,0)--(-2.67,2.06) node[right]{};
				\draw[blue, ->, line width=0.7pt, dotted] (0,0)--(2.67,-2.06) node[right]{};
				\draw[blue, ->, line width=0.7pt, dotted] (0,0)--(3.1,1.181) node[right]{};
				\draw[blue, ->, line width=0.7pt, dotted] (0,0)--(3.1,-1.181) node[right]{};
				\draw[blue, ->, line width=0.7pt, dotted] (0,0)--(-3.1,1.181) node[right]{};
				\draw[blue, ->, line width=0.7pt, dotted] (0,0)--(-3.1,-1.181) node[right]{};
				\def\x{0}
				\def\y{0}
				\node[anchor = south, font=\scriptsize] at (\x+4.7, \y+2.0) {$R_{\alpha_1(Q)}$};
				\node[anchor = south, font=\scriptsize] at (\x-4.5, \y-2.9) {$R_{\alpha_4(Q)}$};
				\node[anchor = south, font=\scriptsize] at (\x+4.7, \y-2.9) {$R_{\alpha_0(Q)}$};
				\node[anchor = south, font=\scriptsize] at (\x-4.5, \y+2.3) {$R_{\alpha_3(Q)}$};
				\node[anchor = south, font=\scriptsize] at (\x+4.0, \y+0.7) {$R_{\alpha_1(Q)-\frac{\epsilon}{3}}$};
				\node[anchor = south, font=\scriptsize] at (\x+2.9, \y+2.0) {$R_{\alpha_1(Q)+\frac{\epsilon}{3}}$};
				\node[anchor = south, font=\scriptsize] at (\x+0.9, \y+2.9) {$R_{\alpha_2(Q)-\frac{\epsilon}{3}}$};
				\node[anchor = south, font=\scriptsize] at (\x-0.9, \y+2.9) {$R_{\alpha_2(Q)+\frac{\epsilon}{3}}$};
				\node[anchor = south, font=\scriptsize] at (\x-2.5, \y+2.0) {$R_{\alpha_3(Q)-\frac{\epsilon}{3}}$};
				\node[anchor = south, font=\scriptsize] at (\x-3.9, \y+0.9) {$R_{\alpha_3(Q)+\frac{\epsilon}{3}}$};
				\node[anchor = south, font=\scriptsize] at (\x-2.5, \y-2.7) {$R_{\alpha_4(Q)+\frac{\epsilon}{3}}$};
				\node[anchor = south, font=\scriptsize] at (\x-4.0, \y-1.5) {$R_{\alpha_4(Q)-\frac{\epsilon}{3}}$};
				\node[anchor = south, font=\scriptsize] at (\x-0.9, \y-3.6) {$R_{\alpha_5(Q)-\frac{\epsilon}{3}}$};
				\node[anchor = south, font=\scriptsize] at (\x+0.9, \y-3.6) {$R_{\alpha_5(Q)+\frac{\epsilon}{3}}$};
				\node[anchor = south, font=\scriptsize] at (\x+2.5, \y-2.6) {$R_{\alpha_0(Q)-\frac{\epsilon}{3}}$};
				\node[anchor = south, font=\scriptsize] at (\x+4.0, \y-1.5) {$R_{\alpha_0(Q)+\frac{\epsilon}{3}}$};
				\node[anchor = south, font=\scriptsize] at (\x+0.1, \y+4.0) {$R_{\alpha_2(Q)}$};
				\node[anchor = south, font=\scriptsize] at (\x+0.1, \y-4.6) {$R_{\alpha_5(Q)}$};
				\draw[<->] (1.75,0.7) arc[start angle=20, end angle=-20, radius=2];
				\draw[<->] (1.42,1.17) arc[start angle=35, end angle=78, radius=1.83];
				\draw[<->] (-0.30,1.9) arc[start angle=90, end angle=151, radius=1.44];
				\draw[<->] (-1.78,0.72) arc[start angle=150, end angle=209, radius=1.4];
				\draw[<->] (-1.51,-1.26) arc[start angle=220, end angle=265, radius=1.8];
				\draw[<->] (0.3,-1.9) arc[start angle=280, end angle=323, radius=2];
				\node[anchor = south, font=\scriptsize] at (\x+2.4, \y-0.4) {$\overline{B^\epsilon_0(Q)}$};
				\node[anchor = south, font=\scriptsize] at (\x-2.5, \y-0.4) {$\overline{B^\epsilon_3(Q)}$};
				\node[anchor = south, font=\scriptsize] at (\x+1.2, \y+1.6) {$\overline{B^\epsilon_1(Q)}$};
				\node[anchor = south, font=\scriptsize] at (\x-1.2, \y+1.6) {$\overline{B^\epsilon_2(Q)}$};
				\node[anchor = south, font=\scriptsize] at (\x+1.2, \y-2.4) {$\overline{B^\epsilon_5(Q)}$};
				\node[anchor = south, font=\scriptsize] at (\x-1.2, \y-2.4) {$\overline{B^\epsilon_4(Q)}$};
			\end{tikzpicture}
		\end{center}
		\caption{{\label{Fig2}
				Fundamental partition of the plane with respect to a polynomial $Q$ with degree $m=3$ and with positive leading coefficient. The six sectors $A_j(Q), j=0,1,2,3,4,5$  as defined in (\ref{A_j(P)}) are bounded by the solid rays, whereas the corresponding closed sub-sectors $\overline{B_j(Q)}$ are bounded by dotted (blue) rays.  For each $j=0, 1,2,3,\ldots ,2m-1$,   $R_{\alpha_j(Q)}=\{z:Arg(z)=\alpha_j(Q)\} $ is called a fundamental ray.  }}
	\end{figure}
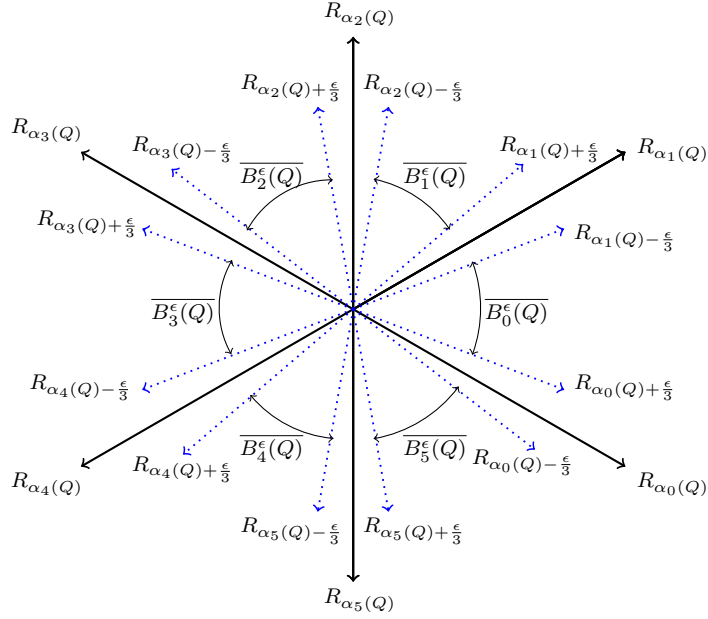

	\begin{lemma}\label{Growth-exppoly}
		Let $Q$ be a polynomial with degree $m$ and leading coefficient $b_m$. For $0<\epsilon <\frac{\pi}{4}$, let $\overline{B^\epsilon_j(Q)}$  be defined as in Equation~(\ref{closed subsector}) for  $j=0,1,2,\ldots,2m-1$. Then, there exists $r_\epsilon >0$ such that the following statements are true for all $z \in \overline{B^\epsilon_j(Q)}$ with $|z|>r_\epsilon$.
		\begin{enumerate}
			\item If $j$ is even then 	$|\exp(Q(z))|>\exp\left(|b_m||z|^m(1-\epsilon)\sin\epsilon \right)$.
			\item  If $j$ is odd then 	$|\exp(Q(z))|<\exp\left(-|b_m||z|^m(1+\epsilon)\sin\epsilon \right)$.
		\end{enumerate} In particular, 
		$|\exp(Q(z))| \to \infty$ or $0$ as $z \to \infty$ in $\overline{B^\epsilon_j(Q)}$ for $j$ even or odd, respectively.
	\end{lemma} 
	\begin{proof}
		Let $0<\epsilon <\frac{\pi}{4}$ and  $Q(z)=b_{0}+b_{1}z +\cdot\cdot\cdot+b_{m}{z }^m$, where $b_m \neq 0$. If $b_k=|b_k| e^{i\beta_k}$ and $z=re^{i\theta}$ then $Q(z)=\sum_{k=0}^{m}|b_k| r^ke^{i(\beta_k+k\theta)}$ and therefore, $|\exp(Q(z))|=\left|\exp\left(\sum_{k=0}^{m}|b_k| r^ke^{i(\beta_k+k\theta)}\right)\right|$, which is nothing but $\exp\left(\sum_{k=0}^{m}|b_k|r^k\cos{(\beta_k+k\theta)}\right)$. This gives that
		\begin{equation}\label{|e^P|}
			|\exp(Q(z))|=\exp\left(|b_m| r^m\cos{(\beta_m+m\theta)}\left(1+\sum_{k=0}^{m-1}\frac{|b_k|\cos{(\beta_k+k\theta)}}{|b_m| r^{m-k}\cos{(\beta_m+m\theta)}}\right)\right).
		\end{equation}
		For every $z=r e^{i\theta} \in\overline{B^\epsilon_j(Q)}$, we have $$  \beta_m+m\theta \in  I_j := (j\pi-\frac{\pi}{2}+\epsilon, j\pi+\frac{\pi}{2}-\epsilon).$$ 
		It is straightforward to verify that $ \cos x>0$ for all $x \in I_j$ and even $j$, whereas $\cos x<0$ for all $x \in I_j$ and odd $j$.
		
		For $j$ even, $\cos x$ increases in $(j\pi-\frac{\pi}{2}+\epsilon,j\pi)$ and then decreases in $(j\pi,j\pi+\frac{\pi}{2}-\epsilon)$. If $j\pi-\frac{\pi}{2}+\epsilon < \beta_m+m\theta \leq j\pi$ then $\cos(\beta_m+m\theta)>  \cos(j\pi-\frac{\pi}{2}+\epsilon)=\sin\epsilon$. If $j\pi < \beta_m+m\theta < j \pi +\frac{\pi}{2} -\epsilon $ then $\cos(\beta_m+m\theta)> \cos(j\pi+\frac{\pi}{2}-\epsilon)=\sin\epsilon$ (see Figure ~\ref{cosine}).
		Therefore, we have
		\begin{equation}\label{even inequality}
			\cos(\beta_m+m\theta) >  \sin\epsilon ~\mbox{whenever $\beta_m+m\theta \in I_j$ for an even}~  j.
		\end{equation}
		\par Similarly,  for odd $j$, $\cos x$ decreases in $(j\pi-\frac{\pi}{2}+\epsilon,j\pi)$ and then increases in $(j\pi,j\pi+\frac{\pi}{2}-\epsilon)$. If $\beta_m+m\theta \leq j\pi$ then $\cos(\beta_m+m\theta)<  \cos(j\pi-\frac{\pi}{2}+\epsilon)=-\sin\epsilon$. If $\beta_m+m\theta > j\pi$ then $\cos(\beta_m+m\theta)< \cos(j\pi+\frac{\pi}{2}-\epsilon)=-\sin\epsilon$ (see Figure~\ref{cosine}), and we have
		\begin{equation}\label{odd inequality}
			\cos(\beta_m+m\theta) < -\sin\epsilon ~\mbox{whenever $\beta_m+m\theta \in I_j$ for an odd}~  j.  
		\end{equation}
		\par  It follows from Equations (\ref{even inequality}) and (\ref{odd inequality}) that $|\cos(\beta_m+m\theta)|\geq \sin\epsilon$ for all $j$. Thus, if $z\in \overline{B^\epsilon_j(Q)}$, then
		\begin{equation}\label{Estimate}
			\left|\sum_{k=0}^{m-1}\frac{|b_k|\cos{(\beta_k+k\theta)}}{|b_m|\cos{(\beta_m+ m\theta)r^{m-k}}}\right|\leq  \sum_{k=0}^{m-1}\frac{|b_k||\cos{(\beta_k+k\theta)}|}{|b_m||\cos{(\beta_m+ m\theta) |r^{m-k}}}
			\leq\frac{1}{\sin\epsilon}\sum_{k=0}^{m-1}\frac{|b_k|}{|b_m|r^{m-k}}.
		\end{equation}
		The right-hand side of the Inequality (\ref{Estimate}) can be made sufficiently small by taking $r\to+\infty$. In other words, for the same $\epsilon>0$, taken earlier, there exists $r_\epsilon>0$ such that \begin{equation}\label{Estimate < epsilon}
			\left|\sum_{k=0}^{m-1}\frac{|b_k|\cos{(\beta_k+k\theta)}}{|b_m|\cos{(\beta_m+m\theta)r^{m-k}}}\right|<\epsilon ~\mbox{for all } r>r_\epsilon.
		\end{equation}
		Thus, we have the following.
		\begin{enumerate}
			\item If $z\in\overline{B^\epsilon_j(Q)}$ for an even $j$, then it follows from Equation (\ref{|e^P|}) and Inequality (\ref{even inequality}) that
			$$	|\exp(Q(z))| \geq \exp \left( |b_m|  r^m  \left(1+\sum_{k=0}^{m-1}\frac{|b_k|\cos{(\beta_k+k\theta)}}{|b_m| r^{m-k}\cos{(\beta_m+m\theta)}}\right)\sin\epsilon\right).$$
			The sum in the right-hand side is larger than $-\epsilon$ by Inequality~(\ref{Estimate < epsilon}). Thus, noting that $|z|=r$, we have 	
			\begin{equation} \label{e^P-unbounded}
				|\exp(Q(z))|>\exp \left(|b_m| |z|^m  (1-\epsilon )\sin\epsilon \right)~\mbox{for all}~ |z|> r_\epsilon.
			\end{equation} 
			
			\item If $z\in\overline{B^\epsilon_j(Q)}$ for an odd $j$  then it follows from Equation (\ref{|e^P|}) and Inequality (\ref{odd inequality})  that
			$$	|\exp(Q(z))| \leq \exp \left( |b_m| r^m  \left(1+\sum_{k=0}^{m-1}\frac{|b_k|\cos{(\beta_k+k\theta)}}{|b_m| r^{m-k}\cos{(\beta_m+m\theta)}}\right)(-\sin\epsilon) \right).$$
			The sum in the right-hand side is larger than $-\epsilon$ by Inequality~(\ref{Estimate < epsilon}). Thus, noting that $|z|=r$, we have 	
			\begin{equation}\label{e^P-bounded}
				|\exp(Q(z))|<\exp\left(-|b_m||z|^m(1-\epsilon)\sin\epsilon\right) 
				~\mbox{for all}~ |z|> r_\epsilon.\end{equation} 
		\end{enumerate}
		\begin{figure}[ht]
			\centering
			\begin{tikzpicture}[scale=1.1]
				\draw[<->] (-2.8,0) -- (2.8,0);
				\draw[<->, dotted] (-2.8,-1) -- (2.8,-1);
				\draw[<->, dotted] (-2.8,1) -- (2.8,1);
				% 		\draw[<->] (0,-1.3) -- (0,1.3);
				\draw[thick,domain=-pi/2:pi/2,samples=100]
				plot ({2*\x/pi}, {cos(\x r)});
				\fill (-1,0) circle (1.2pt);
				\fill (1,0) circle (1.2pt);
				\fill (0,0) circle (1.2pt);
				\node[below, font=\tiny] at (-1,0) {$j\pi-\frac{\pi}{2}$};
				\node[below, font=\tiny] at (1,0) {$j\pi+\frac{\pi}{2}$};
				\node[below, font=\tiny] at (0,0) {$j\pi$};
			\end{tikzpicture}
			\begin{tikzpicture}[scale=1.1]
				\draw[<->] (-2.8,0) -- (2.8,0);
				\draw[<->, dotted] (-2.8,-1) -- (2.8,-1);
				\draw[<->, dotted] (-2.8,1) -- (2.8,1);
				% 			\draw[<->] (0,-1.3) -- (0,1.3);
				\draw[thick,domain=-pi/2:pi/2,samples=100]
				plot ({2*\x/pi}, {-cos(\x r)});
				\fill (-1,0) circle (1.2pt);
				\fill (1,0) circle (1.2pt);
				\fill (0,0) circle (1.2pt);
				\node[above, font=\tiny] at (-1,0) {$j\pi-\frac{\pi}{2}$};
				\node[above, font=\tiny] at (1,0) {$j\pi+\frac{\pi}{2}$};
				\node[above, font=\tiny] at (0,0) {$j\pi$};
			\end{tikzpicture}
			\caption{\label{cosine}Graphs of cosine on
				$\left(j\pi-\frac{\pi}{2},j\pi+\frac{\pi}{2}\right)$  
				for even (left image) and odd  (right image) $j$.}
		\end{figure}
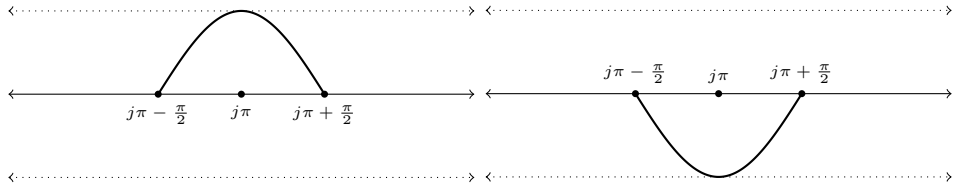

		Since $0< \epsilon < \frac{\pi}{4}$,  $(1-\epsilon)\sin \epsilon $ is positive. It now follows from  Lemma~\ref{Growth-exppoly} that  $|\exp(Q(z))| \to \infty$ or $0$ as $z \to \infty$ in $\overline{B^\epsilon_j(Q)}$ for $j$ even or odd, respectively.
	\end{proof}
	In the following lemma, for a curve $\gamma$, the set $\{\Im(Q(z)):z\in\gamma\}$ is denoted by $\Im(Q(\gamma))$, and $\mathcal{O}(z^k)$ denotes a function of $z$ such that $\lim_{z\to\infty}\frac{\mathcal{O}(z)}{z^{k}}$ is bounded.
	
	\begin{lemma}\label{Imaginary_part_fundamental_ray}
		Let $Q$ be a polynomial of degree $m$ and $\beta_m$ be the argument of the leading coefficient of $Q$. Define    $\alpha_j(Q)=-\frac{\beta_m}{m}-\frac{\pi}{2m}+\frac{\pi j}{m}$ for $j=0, 1,2, \ldots,2m-1$. If $\gamma:[0,\infty)\to\mathbb{C}$ is an unbounded curve
		containing a sequence $\{z_n\}_{n\geq 1}$  such that $\lim_{n \to \infty}z_n = \infty$ and $\lim_{n \to \infty}Arg(z_n) =\alpha_{j}(Q)$  for some $j$, then $\Im(Q(\gamma))$ contains either $(l, \infty)$ or $(-\infty, -l)$ for some $l>0$ whenever $j$ is odd or even, respectively. Further, if $\widetilde{Q}$ is a polynomial with degree less than  $m$, then the same conclusion holds for $\Im((Q -\widetilde{Q} )(\gamma))$.
	\end{lemma}
	
	\begin{proof}
		Let $z_n=r_ne^{i\theta_n}$ and $j\in\{0,1,2,\ldots,2m-1\}$ be such that $z_n\to\infty$ and $\theta_n\to\alpha_{j}(Q)$ as $n\to\infty$. Denoting  the leading coefficient of $Q$ by $b_m$, we have $Q(z)=b_mz^m+\mathcal{O}(z^{m-1})$ and therefore,
		$$
		\Im(Q(z_n))=\Im(b_m z^m)+\Im(\mathcal{O}(z^{m-1})) =|b_m|r_n^m\sin(\beta_m+m\theta_n)+\mathcal{O}(r_n^{m-1}).
		$$
		On the other hand, by the definition of $\alpha_{j}(Q)$,
		$\beta_m+m\alpha_{j}(Q)=-\frac{\pi}{2}+\pi j$
		and hence $\sin(\beta_m+m\alpha_{j}(Q))=(-1)^{j+1}$. Since $\theta_n\to\alpha_{j}(Q)$, we have $\sin(\beta_m+m\theta_n)\to(-1)^{j+1}$ as $n\to\infty$, and consequently for all sufficiently large $n$,
		$$\Im(Q(z_n))=(-1)^{j+1}|b_m|r_n^m(1+ o(1)), ~\mbox{where}~ o(1) ~\mbox{ is a function of }~z_n~\mbox{such that}~\lim_{n \to \infty} o(1) =0. $$
		As $r_n\to\infty$, we have
		$\Im(Q(z_n))\to+\infty~\text{if }~j\text{ is odd}, ~\text{and}~ \Im(Q(z_n))\to-\infty
		~ \text{if }j\text{ is even}$.
		
		\par Since $z_n\in\gamma$ for every $n$, there exists a sequence $\{t_n\}_{n\geq 1}$ in $[0,\infty)$ such that
		$\gamma(t_n)=z_n$. Clearly, the function $h:[0,\infty)\to\mathbb{R}$ defined by $h(t)=\Im(Q(\gamma(t)))$ is continuous. This gives that the set $\Im(Q(\gamma))$ is a connected subset of $\mathbb{R}$. Moreover, by virtue of
		$h(t_n)=\Im(Q(\gamma(t_n)))=\Im(Q(z_n))$, this set is unbounded. More specifically, $\Im(Q(\gamma))$ contains either $(l, \infty)$ or $(-\infty, -l)$ for some $l>0$ whenever $j$ is odd or even, respectively.
		\par 
		The degree and the leading coefficient of the polynomial $Q -\widetilde{Q}$ are the same as those of  $Q$. Therefore $\alpha_j(Q)= \alpha_j(Q-\widetilde{Q})$ for each $j$, and the fundamental partitions with respect to $Q$ and $Q-\widetilde{Q}$ are identical. Thus, the conclusion of the previous paragraph holds for $Q-\widetilde{Q}.$
	\end{proof}
	
	%%%%%%%%%
	\par The next lemma plays a crucial role in the proof of Theorems \ref{e^{Q_1 }+e^{Q_2 }+P_3} and \ref{Generalization}. We first introduce the necessary notions required for its statement.
	For a non-constant polynomial $Q$ with degree $m$ and an unbounded curve $\gamma: [0, \infty) \to \mathbb{C}$, the position of $\gamma$ with respect to the fundamental partition of $Q$ is a key ingredient for our proofs. For  $j=0,1,2,\ldots, 2m-1,$ from Equation (\ref{closed subsector}) recall that $	\overline{B^\epsilon_j(Q)}=\left\{z:\alpha_j(Q)+\frac{\epsilon}{m}\leq Arg(z)\leq\alpha_{j+1}(Q)-\frac{\epsilon}{m}\right\}$ for positive $\epsilon$. 
	We say the following about the position of $\gamma$.
	\begin{enumerate}
		\item \label{odd-indexed} $\gamma$ is \textbf{unbounded in an odd-indexed closed sub-sector} of $Q$ if the set $\gamma \cap \overline{B_j ^\epsilon(Q)}$ is unbounded for some positive $\epsilon$ and some odd $j$.
		\item \label{even-indexed} $\gamma$ is \textbf{unbounded in an even-indexed closed sub-sector} of $Q$ if the set $\gamma \cap \overline{B_j ^\epsilon(Q)}$ is unbounded for some  positive $\epsilon$ and some even $j$.
		\item \label{asymptotic}
		$\gamma$ is \textbf{asymptotic to a fundamental ray} of $Q$ if none of the previous possibilities (1) or (2) holds.  In this case, there is at least one  $j \in\{0,1,2,\ldots,2m-1\}$ such that, for $S_{\alpha_j (Q)}^\epsilon:=
		\{z: \alpha_j (Q) -\epsilon< Arg(z)< \alpha_j (Q) +\epsilon\}$, 
		the set $\gamma \cap S_{\alpha_j (Q)}^\epsilon $ is unbounded but the intersection of $  \gamma $ with the boundary of $  S_{\alpha_j (Q)}^\epsilon $ is bounded  for every   $ \epsilon \in (0, \frac{\pi}{2m})$. In this case, we say (for the sake of precision) $\gamma$ is asymptotic to the ray $R_{\alpha_j (Q)}$. Consequently, for every fixed $\epsilon_0 \in (0, \frac{\pi}{2m})$, exactly one of the following holds:
		\begin{enumerate}
			\item $\{\gamma(t): t>t_0\} \subset S_{\alpha_j (Q)}^{\epsilon_0}$ for some    $t_0 >0$. Let $\gamma^*$ denote the set  $\{\gamma(t): t>t_0\}$.
			In this case, $\gamma$ cannot be asymptotic to any fundamental ray corresponding to an angle other than $\alpha_j (Q)$ (see the curves $\gamma_{2}, \gamma_3$ and $\gamma_4$ in Figure~\ref{curve-position}). 
			\item  there is  $M>0$ such that the union of the two sets $ \gamma \cap S_{\alpha_j (Q)}^{\epsilon_0} \cap \{z: |z| > M\}$ and $ \{z: |z|=M\}$ is connected. If $(a)$ does not hold, we denote this connected set by $\gamma^*$  (see the curve $\gamma_5$ and $\gamma_{5}^*$ in Figure~\ref{curve-position}). 
			%Since $\gamma^* \cap \gamma$ is unbounded, for every entire function $f$, $f(\gamma)$ is unbounded whenever $f(\gamma^*)$ is unbounded.
		\end{enumerate}   
		A complete demonstration of the above situations is given in Figure~\ref{curve-position}. The curve $\gamma^*$ has the following key property whenever $\gamma$ is asymptotic to the ray $R_{\alpha_j(Q)}$:
		\begin{equation}\label{gamma-star}
			\parbox{0.9\textwidth}{
				$\gamma\cap\gamma^*$ is unbounded in such a way that if $z_n\in\gamma^*$ is a sequence with $\lim_{ n\to \infty}z_n=\infty$, then all but finitely many terms of the sequence lie on $\gamma$, and satisfy
				$\displaystyle\lim_{n\to\infty}Arg(z_n)=\alpha_j(Q)$.}
		\end{equation}			
	\end{enumerate}
	
	It is important to note that in the case \ref{asymptotic}(b), $\lim_{t \to \infty} Arg(\gamma(t))$ may not exist. There may exist an unbounded sequence $\{z_n\}_{n\geq 1}$ on $\gamma$ converging to $\infty$ such that it has two subsequences $\{z_{n_k}\}_{k\geq 1}$ and $\{z_{m_k}\}_{k\geq 1}$, also converging to $\infty$ satisfying  $\lim_{k \to \infty} Arg(z_{n_k}) = \alpha_{j }(Q)$ and $\lim_{k \to \infty} Arg(z_{m_k}) = \alpha_{j' }(Q)$  where $j \neq j'$  (see the curve $\gamma_{01}$ in Figure~\ref{curve-position}). 
	
	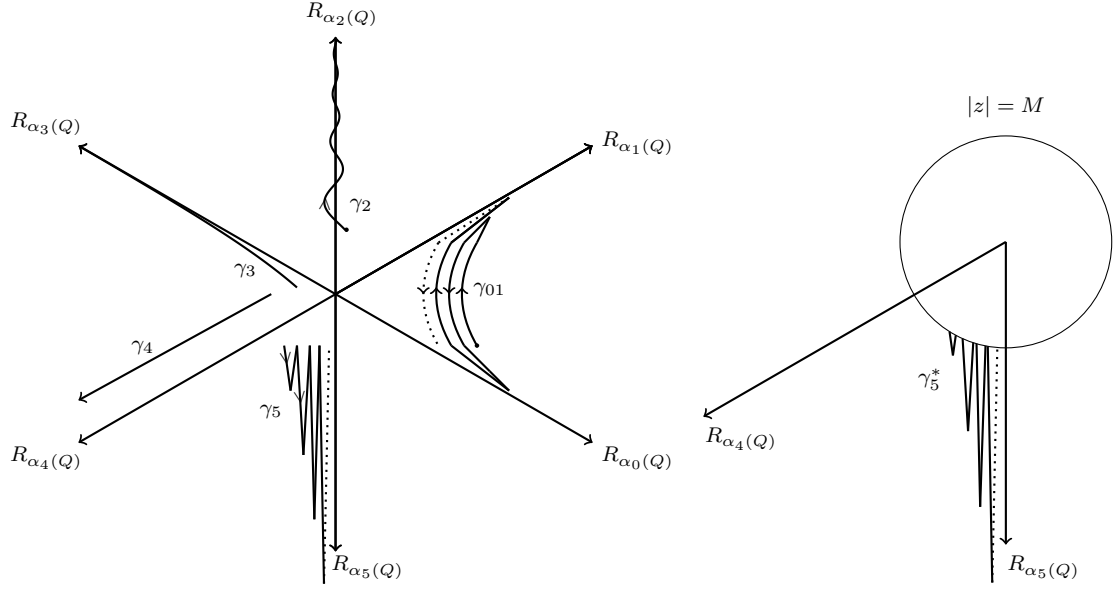
\begin{figure}
		\begin{center}
			\begin{tikzpicture}[scale=0.85]
				\fill[white!30]
				(0,0) -- (3,1.732) -- (0,3) -- cycle;
				\fill[white!30]
				(0,0) -- (-3,1.732) -- (-3,-1.732) -- cycle;
				\fill[white!30]
				(0,0) -- (0,-3) -- (3,-1.732) -- cycle;
				\fill[white!30]
				(0,0) -- (3,1.732) -- (3,-1.732) -- cycle;
				\fill[white!30]
				(0,0) -- (0,3) -- (-3,1.732) -- cycle;
				\fill[white!30]
				(0,0) -- (0,-3) -- (-3,-1.732) -- cycle;
				\draw[<->,thick] (0,-4)--(0,4) node[above]{};
				\draw[black, ->,thick] (0,0)--(4,2.309) node[right]{};
				\draw[black, ->,thick] (0,0)--(4,-2.309) node[right]{};
				\draw[black, ->,thick] (0,0)--(4,2.309) node[right]{};
				\draw[black, ->,thick] (0,0)--(4,2.309) node[right]{};
				\draw[black, ->,thick] (0,0)--(-4,-2.309) node[right]{};
				\draw[black, ->,thick] (0,0)--(-4,2.309) node[right]{};
				\draw[black, ->,thick] (0,0)--(0,4) node[right]{};
				\draw[black, ->,thick] (0,0)--(0,-4) node[right]{};
				\def\x{0}
				\def\y{0}
				\node[anchor = south, font=\scriptsize] at (\x+4.7, \y+2.0) {$R_{\alpha_1(Q)}$};
				\node[anchor = south, font=\scriptsize] at (\x-4.5, \y-2.9) {$R_{\alpha_4(Q)}$};
				\node[anchor = south, font=\scriptsize] at (\x+4.7, \y-2.9) {$R_{\alpha_0(Q)}$};
				\node[anchor = south, font=\scriptsize] at (\x-4.5, \y+2.3) {$R_{\alpha_3(Q)}$};
				\node[anchor = south, font=\scriptsize] at (\x+0.1, \y+4.0) {$R_{\alpha_2(Q)}$};
				\node[anchor = south, font=\scriptsize] at (\x+0.5, \y-4.6) {$R_{\alpha_5(Q)}$};
				\node[anchor = south, font=\scriptsize] at (\x+2.4, \y-0.2) {$\gamma_{01}$};
				\fill (2.2,-0.8) circle (1pt);
				\draw[
				thick,
				postaction={
					decorate,
					decoration={
						markings,
						mark=at position 0.5 with {\arrow{<}}
					}
				}
				] (2.2,0.8) to[bend right=30] (2.2,-0.8);
				\draw[thick] (2.2,0.8) -- (2.4,1.2);
				\draw[
				thick,
				postaction={
					decorate,
					decoration={
						markings,
						mark=at position 0.5 with {\arrow{>}}
					}
				}
				] (2.0,0.8) to[bend right=30] (2.0,-0.8);
				\draw[thick] (2.0,0.8) -- (2.4,1.2);
				\draw[thick] (2.0,-0.8) -- (2.4,-1.2);
				\draw[
				thick,
				postaction={
					decorate,
					decoration={
						markings,
						mark=at position 0.5 with {\arrow{<}}
					}
				}
				] (1.8,0.8) to[bend right=30] (1.8,-0.8);
				\draw[thick] (1.8,0.8) -- (2.7,1.5);
				\draw[thick] (1.8,-0.8) -- (2.7,-1.5);
				\draw[thick] (2.4,-1.2) -- (2.7,-1.5);
				\draw[dotted, thick] (1.6,0.8) -- (2.7,1.5);
				\draw[dotted,
				thick,
				postaction={
					decorate,
					decoration={
						markings,
						mark=at position 0.5 with {\arrow{>}}
					}
				}
				] (1.6,0.8) to[bend right=30] (1.6,-0.8);
				\draw[thick] (1.8,0.8) -- (2.7,1.5);
				%%%%%%%%%%%%%%%%%%%%%%%%%%%%%
				\draw[
				thick,
				samples=3000,
				domain=1:3.95,
				smooth
				]
				plot (
				{0.55*exp(-0.8*\x)*sin(deg(2.5*\x^1.7))},
				{\x}
				);
				\node[anchor = south, font=\tiny] at (\x-0.16, \y+1.1) {${\rotatebox{85}{$>$}}$};
				\fill (0.17,1) circle (1pt);
				\node[anchor = south, font=\scriptsize] at (\x+0.4, \y+1.1) {$\gamma_2$};
				%%%%%%%%%%%%%%%%%%%%%%%%%%%%%%%
				\draw[thick] (-0.8,-0.8) -- (-0.7,-1.5);
				\draw[thick] (-0.7,-1.5) -- (-0.6,-0.8);
				\draw[thick] (-0.6,-0.8) -- (-0.5,-2.5);
				\draw[thick] (-0.5,-2.5) -- (-0.40,-0.8);
				\draw[thick] (-0.40,-0.8) -- (-0.33,-3.5);
				\draw[thick] (-0.33,-3.5) -- (-0.25,-0.8);
				\draw[thick] (-0.25,-0.8) -- (-0.18,-4.5);
				\draw[dotted, thick] (-0.18,-4.5) -- (-0.10,-0.8);
				\node[anchor = south, font=\scriptsize] at (\x-1.4, \y+0.1) {$\gamma_3$};
				%%%%%%%%%%%%%%%%%%%%%%
				\draw[black, thick, ->, domain=0.6:4, samples=200]
				plot ({-\x}, {0.57735*\x - 0.5*exp(-1.2*\x)});
				\draw[black, thick, ->] (-1,0)--(-4,-1.65);
				\node[anchor = south, font=\scriptsize] at (\x-3, \y-1.1) {$\gamma_4$};
				\node[anchor = south, font=\scriptsize] at (\x-1, \y-2.1) {$\gamma_5$};
				\node[anchor = south, font=\tiny] at (\x-0.55, \y-1.9) {${\rotatebox{278}{$>$}}$};
				\node[anchor = south, font=\tiny] at (\x-0.77, \y-1.3) {${\rotatebox{278}{$>$}}$};
				%%%%%%%%%%%%%%%%%%%%
			\end{tikzpicture}
			\begin{tikzpicture}
				\draw[black, ->,thick] (8,0)--(4,-2.309) node[right]{};
				\draw[black, ->,thick] (8,0)--(8,-4) node[right]{};
				\def\x{0}
				\def\y{0}
				\node[anchor = south, font=\scriptsize] at (\x+4.5, \y-2.9) {$R_{\alpha_4(Q)}$};
				\node[anchor = south, font=\scriptsize] at (\x+8.5, \y-4.6) {$R_{\alpha_5(Q)}$};
				\draw (\x+8,\y+0) circle (1.4);
				\draw[thick] (7.255,-1.185) -- (7.3,-1.5);
				\draw[thick] (7.3,-1.5) -- (7.33805,-1.23362);
				\draw[thick] (7.42811,-1.27787) -- (7.5,-2.5);
				\draw[thick] (7.5,-2.5) -- (7.56871,-1.33191);
				\draw[thick] (7.61212,-1.34519) -- (7.66,-3.5);
				\draw[thick] (7.66,-3.5) -- (7.73087,-1.37389);
				\draw[thick] (7.76096,-1.37944) -- (7.82,-4.5);
				\draw[dotted, thick] (7.82,-4.5) -- (7.88713,-1.39544);
				\node[anchor = south, font=\scriptsize] at (\x+7, \y-2.1) {$\gamma^*_5$};
				\node[anchor = south, font=\scriptsize] at (\x+8, \y+1.5) {$|z|=M$};
			\end{tikzpicture}
			\caption{\label{curve-position}Possible positions of an unbounded curve $\gamma$ which is asymptotic to the fundamental rays of a cubic polynomial. The curve $\gamma_{01}$ returns to intersect a circle infinitely often and therefore is asymptotic to the two fundamental rays $R_{\alpha_0 (Q)}$ and $R_{\alpha_1 (Q)}$. The curves $\gamma_2$, $\gamma_3$ and $\gamma_5$ are asymptotic to the fundamental rays $R_{\alpha_2 (Q)}, R_{\alpha_3 (Q)}$ and $R_{\alpha_5 (Q)}$, respectively. Though a bit counterintuitive, the curve $\gamma_4$ is also asymptotic to  $R_{\alpha_4 (Q)}$.}
		\end{center}
	\end{figure}
	The notion of \textit{asymptotic to a fundamental ray} can be extended to unbounded subsets of curves, leading to the same consequence. This is described in the following remark, which we are going to use several times in the proofs of Theorems \ref{e^{Q_1 }+e^{Q_2 }+P_3} and \ref{Generalization}.  
	\begin{Remark}\label{asymptotic-nont-connected}
		Let $\gamma'$ be an unbounded subset of an unbounded curve. We say $\gamma'$ is asymptotic to a fundamental ray of a polynomial $Q$ with degree $m$  if there is at least one  $j \in\{0,1,2,\ldots,2m-1\}$ such that, for $S_{\alpha_j (Q)}^\epsilon:=
		\{z: \alpha_j (Q) -\epsilon< Arg(z)< \alpha_j (Q) +\epsilon\}$, 
		the set $\gamma' \cap S_{\alpha_j (Q)}^\epsilon $ is unbounded but the intersection of $  \gamma' $ with the boundary of $  S_{\alpha_j (Q)}^\epsilon $ is bounded  for every   $ \epsilon \in (0, \frac{\pi}{2m})$. As discussed in \ref{asymptotic} previously, it is possible to construct an unbounded connected set $\gamma^*$ using $\gamma'$ that satisfies Property (\ref{gamma-star}). 
	\end{Remark}
	For stating the next lemma, let $Q_1, Q_2$ be two polynomials with degrees $m_1, m_2$, with $m_1 > m_2$ and with leading coefficients $b_{m_1}, b_{m_2}$ respectively. Denote $b_{m_1}=|b_{m_1}|e^{i\beta_{1}}$, $b_{m_2}=|b_{m_2}|e^{i\beta_{2}}$,   $\alpha_j(Q_1)=-\frac{\beta_1}{{m_1}}-\frac{\pi}{2{m_1}}+\frac{\pi j}{{m_1}}$ for $j\in\{0, 1, 2, \ldots,2{m_1}-1\}$, $\alpha_{j'}(Q_2)=-\frac{\beta_2}{{m_2}}-\frac{\pi}{2{m_2}}+\frac{\pi j'}{{m_2}}$ for $j'\in\{0, 1, 2, \ldots,2{m_2}-1\}$, and 
	$\overline{B^\epsilon_{j'}(Q_2)}=\left\{z:\alpha_{j'}(Q_2)+\frac{\epsilon}{{m_2}}\leq Arg(z)\leq\alpha_{{j'}+1}(Q_2)-\frac{\epsilon}{{m_2}}\right\}$ for some positive $\epsilon$.
	\begin{lemma} \label{asymptotic-to-ray}
		Let $Q_1 $ be a polynomial with degree $m_1$ and   $\gamma:[0,\infty)\to\mathbb{C}$ be an unbounded curve asymptotic to a fundamental ray of $Q_1$. Let that ray correspond to $\alpha_{j^*}(Q_1)$ for some $j^*\in\{0,1,2,\ldots,2{m_1}-1\}$. If $Q_2$  is a polynomial with degree less than  $m_1$, then for any non-zero rational map $R$, there is a sequence $\{u_n\}_{n\geq 1}$ on $\gamma$ satisfying $\lim_{n \to \infty}u_n = \infty$ and $\lim_{n \to \infty}Arg(u_n) =\alpha_{j^*}(Q_1)$  such that $-1$ is not a limit point of the sequence $\{R(u_n)\exp({Q_1(u_n)-Q_2(u_n)})\}_{n\geq 1}$.
	\end{lemma}
	\begin{proof}
		Since $\gamma$ is asymptotic to a fundamental ray of $Q_1$ corresponding to the angle $\alpha_{j^*}(Q_1)$, for a sufficiently small $\delta>0$, possibly with $0<\delta<\frac{\pi}{8}$
		, there exists $M>0$   such that  $$\widetilde{\gamma}:=\{z\in\gamma:|z|>M ~\mbox{and}~\alpha_{j^*}(Q_1)-\delta < Arg(z) < \alpha_{j^*}(Q_1)+\delta \}$$ is an unbounded (but not necessarily connected) subset of $\gamma$. If $\widetilde{\gamma}$ is not connected, we consider its union with the circle $\{z: |z|=M\}$ and denote it by the same  $\widetilde{\gamma}$ for the sake of notational simplicity. Note that, for every sequence $z_n\in\widetilde{\gamma}$, such that $\lim_{n \to \infty}z_n = \infty$, we must have $\lim_{n \to \infty}Arg(z_n) =\alpha_{j^*}(Q_1)$. For such a sequence, applying the last part of Lemma \ref{Imaginary_part_fundamental_ray} to $\widetilde{\gamma}$, we get an unbounded interval contained in
		$\Im((Q_1-Q_2)(\widetilde{\gamma}))$. Without loss of generality, assume that this interval is contained in the positive real axis.

		\par If $R=\frac{P_1}{P_2}$ where $P_1$ and $P_2$ are polynomials with the arguments of their leading coefficients as $p_1$ and $p_2$ respectively, then define $$\eta=\deg(P_1)\alpha_{j^*}(Q_1)+p_1-\deg(P_2)\alpha_{j^*}(Q_1)-p_2~\mbox{and take }~  \theta^*=\frac{\pi}{2}-\eta.$$ 
		
		Choose a sequence  $\{w_n\}_{n\geq 1}$ on $(Q_1-Q_2)(\widetilde{\gamma})$ such that $\Im(w_n)=\theta^* +2(n_0 +n)\pi$ for some $n_0$. This is possible as $\Im((Q_1-Q_2)(\widetilde{\gamma}))$ contains an unbounded interval in the positive real axis. Then, there exist $u_n \in \widetilde{\gamma}$ such that
		$(Q_1-Q_2)(u_n)=w_n$. Since $\lim_{n \to \infty}w_n =\infty$ and $Q_1 -Q_2$ is a polynomial, we have $\lim_{n \to \infty}u_n = \infty$. Further,  $\lim_{n \to \infty}Arg(u_n) =\alpha_{j^*}(Q_1)$ as   $\gamma \bigcap \overline{B^\epsilon_j(Q_1)}$ is bounded for each $j$ and for each $\epsilon>0$.
		
		Now    $\lim_{n \to \infty} Arg (R(u_n)) = \lim_{n \to \infty}\left(Arg(P_1(u_n))-Arg(P_2(u_n))\right) =\eta$, by Corollary \ref{Cor1}.
		For the earlier chosen $\delta$ and for all sufficiently large $n$, we have
		\begin{equation}\label{Arg(R(z) )}
			Arg(R(u_n))=Arg(P_1(u_n))-Arg(P_2(u_n))\in(\eta-\delta, \eta+\delta).
		\end{equation}
		Since $Arg(\exp({Q_1(u_n)-Q_2(u_n)}))=\Im(Q_1(u_n)-Q_2(u_n))=\Im(w_n)=\theta^* +2(n_0 +n)\pi$, we have
		$$Arg\left(R(u_n)\exp({Q_1(u_n)-Q_2(u_n)})\right)=Arg(R(u_n))+\theta^*.$$ 
		It is clear from (\ref{Arg(R(z) )}) that $Arg(R(u_n))+\theta^* \in (\frac{\pi}{2} -\delta, \frac{\pi}{2} +\delta).$  However, by the choice of $\delta$, this interval is contained in $(\frac{\pi}{2} -\frac{\pi}{8}, \frac{\pi}{2}+\frac{\pi}{8})$. In particular, 
		$Arg\left(R(u_n)\exp({Q_1(u_n)-Q_2(u_n)})\right)$  cannot be $\pi$ and consequently,$-1$ is not a limit point of  $\{R(u_n)\exp({Q_1(u_n)-Q_2(u_n)})\}_{n\geq 1}$.
	\end{proof}	
	\section{Proofs}\label{Proofs}
	In this section, we present proofs of Theorems \ref{P_1e^{Q_1}+P_2}, \ref{e^{Q_1 }+e^{Q_2 }+P_3} and \ref{Generalization}, and use the following notations: 
	$P_i(z)=\sum_{k=0}^{d_i}a_{i,k}z^k$ and $Q_i(z)=\sum_{k=0}^{m_i}b_{i,k}z^k$ where $a_{i,d_i}, b_{i,m_i}\neq0$.
	
	We first recall a well-known inequality for polynomials, followed by a simple observation concerning exponential polynomials with real coefficients, which follows directly from elementary calculus.

	\begin{lemma}
		If $P$ is a non-constant polynomial then there exists $M_1, M_2, R>0$ such that  $M_1 |z|^{\deg(P)} < |P(z)|< M_2 |z|^{\deg(P)}$ for all $|z|>R$.
		\label{growth-poly} 
	\end{lemma}
	\begin{lemma}\label{exp-dominating}
		For a natural number  $l$, let $G_l(x)=\sum_{k=1} ^{l} A_{k} x^{m_k} \exp(\widetilde{A_{k}} x^{n_k})$ where  $ A_k, \widetilde{A_{k}} \in \mathbb{R}$ and $m_k, n_k$ are natural numbers for each $k$. If $A_1, \widetilde{A_1}>0$ and $n_1 >n_k$ for $k=2,3,4,\ldots,l$ then $\lim_{x \to +\infty} G_l (x)=+\infty$. 
	\end{lemma}
	\begin{proof}
		Since $A_1,\widetilde A_1>0$ and $n_1>n_k$ for $k\ge2$, putting $C=\max_{2\le k\le l}|\widetilde A_k|$, we have
		$$
		G_l(x)\geq A_1 x^{m_1}\exp({\widetilde A_1x^{n_1}})-\sum_{k=2}^l|A_k|x^{m_k}\exp({Cx^{n_k}}).
		$$ Since $n_1>n_k$ for each $k=2,3,4,\ldots,l$, the lemma follows.
	\end{proof}
	
	%		\begin{proof}
		%		Let $P(z)=a_{0}+a_{1}z +\cdot\cdot\cdot+a_{d}{z }^d$ be a polynomial of degree $d$, where $a_d\neq0$. Then
		%		$$
		%		\lim_{z\to \infty}\frac{P(z)}{z^d}=\lim_{z\to \infty}\left(a_d+\frac{a_{d-1}}{z}+\cdots+\frac{a_0}{z^d}\right)=a_d$$	Hence, there exists $R>0$ such that, for $|z|>R$,
		%		$\frac{|a_d|}{2}<\left|\frac{P(z)}{z^d}\right|<\frac{3|a_d|}{2}$.
		%		Therefore, by taking $M_1=\frac{|a_d|}{2}$ and $M_2=\frac{3|a_d|}{2}$, for all $|z|>R$ we obtain
		%		$M_1|z|^{\deg(P)}<|P(z)|<M_2|z|^{\deg(P)}$.	
		%	\end{proof}
We are now ready to prove Theorems~\ref{P_1e^{Q_1}+P_2} and \ref{e^{Q_1 }+e^{Q_2 }+P_3}.  	\begin{proof}[\bf{Proof of Theorem \ref{P_1e^{Q_1}+P_2}}]
		Let $F_2(z)=P_1(z)\exp(Q_1(z))+P_2(z)$ and $d_i$ be the degree of $P_i$ for $i=1,2$ and $\gamma$ be an unbounded curve. In view of Lemma~\ref{Bov iff}, it is enough to find  an unbounded sequence $\{z_n\}_{n\geq 1}$ on $\gamma$ such that $\{F_2(z_n)\}_{n\geq 1}$ is unbounded. Depending on the position of $\gamma$ with respect to the fundamental partition of the plane with respect to $Q_1$ (see (\ref{A_j(P)})), there are two situations. To describe these, recall that $\overline{B^\epsilon_j(Q_1)}=\left\{z:\alpha_j(Q_1)+\frac{\epsilon}{m_1}\leq Arg(z)\leq\alpha_{j+1}(Q_1)-\frac{\epsilon}{m_1}\right\}$ for $\epsilon>0$ and  $j\in \{0,1,2,\ldots,2m_1-1\}$.
		
		\begin{enumerate}
			\item If the set $\gamma \bigcap \overline{B^\epsilon_j(Q_1)}$ is unbounded for some $\epsilon>0$ and for some $j$, then $\gamma \bigcap \overline{B^\epsilon_j(Q_1)}$ is also unbounded for a smaller value of $\epsilon$. In view of this, assume without loss of generality that $0< \epsilon < \frac{\pi}{4}$, and consider a sequence $\{z_n\}_{n\geq 1}$ in $\gamma \bigcap \overline{B^\epsilon_j(Q_1)}$ such that $\lim_{ n\to \infty}z_n=\infty$.  We have the following two cases depending on the parity of $j$. 
			
			{\bf{Case 1.}} (\textit{Unbounded in an odd-indexed closed sub-sector of $Q_1$}) If $j$ is odd  then  applying Lemma~\ref{Growth-exppoly}(2) to $Q_1$, we have $|\exp(Q_1(z_n))|< \exp\left(-|b_{1, m_1}||z_n|^{m_1}(1+\epsilon)\sin\epsilon \right)$ for all sufficiently large values of $n$. The quantity $B:=|b_{1, m_1}|(1+\epsilon)\sin \epsilon $  is positive by the choice of $\epsilon$. This, along with the triangle inequality, gives that  
			$$|F_2(z_n)| \geq|P_2(z_n)|-|P_1(z_n)\exp(Q_1(z_n))|> |P_2 (z_n)|- |P_1 (z_n)|\exp(-B |z_n|^{m_1}).$$
			It follows from  Lemma~\ref{growth-poly} that the second term in the right-hand side of the above inequality goes to $0$ as $n \to \infty$.	Since $|P_2(z_n)|\to\infty$ as $n\to\infty$,  $|F_2(z_n)|\to\infty$ as $n\to \infty$  and  $\{F_2(z_n)\}_{n\geq 1}$ becomes unbounded.
			
			{\bf{Case 2.}}  (\textit{Unbounded in an even-indexed closed sub-sector of $Q_1$}) If $j$ is even  then applying Lemma~\ref{Growth-exppoly}(1) to $Q_1$, we have $|\exp(Q_1(z_n))|>\exp\left(|b_{1, m_1}||z_n|^{m_1}(1-\epsilon)\sin\epsilon \right)$ for all sufficiently large values of $n$. The quantity $A:=|b_{1, m_1}|(1-\epsilon)\sin \epsilon $  is positive by the choice of $\epsilon$. This, along with the triangle inequality, gives that 
			$$|F_2(z_n)|\geq |P_1(z_n)\exp(Q_1(z_n))|-|P_2(z_n)|>|P_1(z_n)|\exp(A |z_n|^{m_1})-|P_2(z_n)|.$$
			For all sufficiently large $n$, the right-hand side of the above inequality is larger than  $M_1|z_n|^{d_1} \exp(A{|z_n|^{m_1}})-M_2|z_n|^{d_2}$ for some $M_1, M_2 >0$ (see Lemma~\ref{growth-poly}). Clearly, this gives that $|F_2(z_n)|\to\infty$ as $n\to \infty$  (by Lemma~\ref{exp-dominating}) and consequently, $\{F_2(z_n)\}_{n\geq 1}$ is unbounded.

			\item  (\textit{Asymptotic to a fundamental ray of $Q_1$})
			Let $\gamma  $ be asymptotic to a ray $R_{\alpha_j (Q_1)} $. Then there is an unbounded curve $\gamma^*$ satisfying Property (\ref{gamma-star}). Since there is a sequence $\{u_n\}_{n\geq 1}$ on $\gamma^*$ such that $\lim_{n\to\infty}u_n=\infty$ and $\lim_{n\to\infty}Arg(u_n)=\alpha_{j}(Q_1)$, it follows from Lemma \ref{Imaginary_part_fundamental_ray} that $\Im{(Q_1(\gamma^*))}$ contains an unbounded interval and without loss of generality assume that this interval is $(l,+\infty)$ for some $l>0$.
			By virtue of this, we can choose a sequence  $\{w_n\}_{n\geq 1}$ on $Q_1(\gamma)$ such that 
			\begin{equation}\label{W_n=W_0+2npi}
				\Im(w_n)=\widetilde{\beta_2}-\widetilde{\beta_1}+(d_2-d_1)\alpha_{j}(Q_1)+2 (n_0 +n)\pi, ~\mbox{for some}~n_0\in\mathbb{N}
			\end{equation} where $\widetilde{\beta_1}$ and $\widetilde{\beta_2}$ are the arguments of the leading coefficients of $P_1$ and $P_2$ respectively. Thus, each point  $\exp(w_n)$ lies on the ray $R_{\overline{\theta}}$ where $$\overline{\theta} = \beta_2-\beta_1+(d_2-d_1)\alpha_{j}(Q_1).$$ Now, consider a sequence of points $\{z_n\}_{n\geq 1}$ on $\gamma^*$ such that $Q_1(z_n)=w_n$. Clearly,  $\lim_{n\to\infty}z_n=\infty$. Furthermore, by Property~(\ref{gamma-star}) we have $\lim_{n\to\infty}Arg(z_n)=\alpha_{j}(Q_1)$.
			
			\par Next, we show that $\{F_2(z_n)\}_{n\geq 1}$ is unbounded. Applying Corollary \ref{Cor1} to $P_1$ and $P_2$,  and taking $\theta =\alpha_{j}(Q_1)$, for any chosen $0<\delta<\frac{\pi}{8}$, we get $N\in\mathbb{N}$ such that for all $n\geq N$,
			\begin{equation}\label{Arg(P1)}
				Arg(P_1(z_n))\in\left(\widetilde{\beta_1}+d_1\alpha_{j}(Q_1)-\delta,\widetilde{\beta_1}+d_1\alpha_{j}(Q_1)+\delta\right)
			\end{equation} and
			\begin{equation}\label{Arg(P2)}
				Arg(P_2(z_n))\in\left(\widetilde{\beta_2}+d_2\alpha_{j}(Q_1)-\delta,\widetilde{\beta_2}+d_2\alpha_{j}(Q_1)+\delta\right).
			\end{equation}
			Letting  $S_{P_1}:=\{z:\widetilde{\beta_1}+d_1\alpha_{j}(Q_1)-\delta<Arg(z)<\widetilde{\beta_1}+d_1\alpha_{j}(Q_1)+\delta\}$  and  $S_{P_2}:=\{z:\widetilde{\beta_2}+d_2\alpha_{j}(Q_1)-\delta<Arg(z)<\widetilde{\beta_2}+d_2\alpha_{j}(Q_1)+\delta\} $, observe that $$P_i(z_n) \in S_{P_i}~\mbox{for}~ i =1,2 ~\mbox{and for all}~ n\geq N.$$

			\par The fact  $Arg(P_1(z_n)\exp({Q_1{(z_n)}}))=Arg(P_1(z_n))+\overline{\theta} $  along with the   (\ref{Arg(P1)}) gives that  $$Arg(P_1(z_n))+\overline{\theta} \in \left(\widetilde{\beta_1}+d_1\alpha_{j}(Q_1)-\delta+\overline{\theta} ,\widetilde{\beta_1}+d_1\alpha_{j}(Q_1)+\delta+ \overline{\theta} \right).$$ A simple calculation yields $\widetilde{\beta_1}+d_1\alpha_{j}(Q_1)-\delta+\overline{\theta}=\widetilde{\beta_2}+d_2\alpha_{j}(Q_1)-\delta$ and $\widetilde{\beta_1}+d_1\alpha_{j}(Q_1)+\delta+\overline{\theta}=\widetilde{\beta_2}+d_2\alpha_{j}(Q_1)+\delta$. Thus, it follows from   (\ref{Arg(P2)})  that for all $n\geq N$, $P_1(z_n)\exp({Q_1{(z_n)}})\in S_{P_2}$. The opening of $S_{P_2}$ is $2\delta$, that is less than $\frac{\pi}{4}$ by our choice of $\delta$. 
			Since the sequence $\{-P_2(z_n)\}_{n\geq 1}$ is unbounded and is contained in $-S_{P_2}:=\{-z:z\in S_{P_2}\}$ and the opening of the sector $S_{P_2}$ (therefore of $-S_{P_2}$) is less than $\frac{\pi}{4}$, the distance $|P_1(z_n) \exp({Q_1{(z_n)}})-(-P_2(z_n))|$ goes to $\infty$  as $n \to \infty$. In other words, 
			$$\lim_{n \to \infty}|F_2(z_n)|=\infty.$$   Note that the boundedness of $P_1 (z_n) \exp(Q_1 (z_n))$ does not matter.  Hence, $\{F_2(z_n)\}_{n \geq 1}$ is unbounded.	\end{enumerate}
		This completes the proof.
	\end{proof}

	\begin{proof}[\bf{Proof of Theorem \ref{e^{Q_1 }+e^{Q_2 }+P_3}}]
		Let $F_3(z)=P_1(z)\exp(Q_1(z))+P_2(z)\exp(Q _2(z))+P_3(z)$. Since $\deg(Q_1)\neq \deg(Q_2)$, without loss of generality assume that $\deg(Q_1)>\deg(Q_2)$, and consider the fundamental partition with respect to $Q_1$. Let $\gamma$ be an unbounded curve.  In view of the Lemma~\ref{Bov iff}, it is enough to find  an unbounded sequence $\{z_n\}_{n\geq 1}$ on $\gamma$ such that $\{F_3(z_n)\}_{n\geq 1}$ is unbounded. To do it, we consider three situations depending on the position of $\gamma$ with respect to the aforesaid partition. 
		\begin{enumerate}
			\item (\textit{Unbounded in an odd-indexed closed sub-sector of $Q_1$}) If  $ \gamma $ is unbounded in an odd-indexed closed sub-sector of $Q_1$, then there is an odd $j\in\{0,1,2,\ldots,2m_1-1\}$ and a positive $\epsilon$ such that $L :=\gamma \bigcap \overline{B^\epsilon_j(Q_1)}$ is unbounded. For every  sequence $\{u_n\}_{n\geq 1}$ on $L$ such that $\lim_{ n\to \infty}u_n=\infty$, we have $\lim_{n \to \infty} P_1 (u_n) \exp(Q_1 (u_n)) =0$ by  Lemma~\ref{Growth-exppoly}(2). We now consider the fundamental partition with respect to $Q_2$ and the position of $L$ in it.
			\par Let $L' := \overline{B^{\epsilon'}_{j'}(Q_2)} \cap L$ be unbounded  for some $j'\in\{0,1,2,\ldots,2m_2-1\}$ and positive $\epsilon'$. Consider $z_n \in L'$ with $\lim_{n \to \infty} z_n =\infty$.
			\par If $j'$ is odd, then  $\lim_{n \to \infty}P_2(z_n)\exp({Q_2(z_n)})= 0$ by  Lemma~\ref{Growth-exppoly}(2). It follows from $\lim_{n \to \infty} P_3(z_n) =\infty$ and the inequality  $$|F_3(z_n)|\geq |P_3(z_n)|-|P_1(z_n)\exp({Q_1(z_n)})|-|P_2(z_n)\exp({Q_2(z_n)})|,$$ that   $|F_3(z_n)|\to\infty$ as $n\to \infty$.  
			
			\par If $j'$ is even, then we use the inequality  $$|F_3(z_n)|\geq |P_2(z_n)\exp({Q_2(z_n)})|-|P_1(z_n)\exp({Q_1(z_n)})|-|P_3(z_n)|.$$ 
			
			Applying  Lemma~\ref{Growth-exppoly}(1) to $Q_2$ and  Lemma~\ref{growth-poly} to $ P_2 $  and $ P_3$, we find constants $C_2,M_2~\mbox{and}~M_3>0$ such that $|P_2(z_n)\exp({Q_2(z_n)})|\geq M_2|z_n|^{d_2}\exp({C_2{|z_n|^{m_2}}})$ and $|P_3(z_n)|\leq M_3|z_n|^{d_3}$. Putting these estimates in the previous inequality, we get $$|F_3(z_n)|\geq M_2|z_n|^{d_2}\exp({C_2{|z_n|^{m_2}}})-|P_1(z_n)\exp({Q_1(z_n)})|-M_3|z_n|^{d_3}.$$ 
			As $\lim_{n \to \infty} P_1 (z_n) \exp(Q_1 (z_n)) =0$, clearly  $|F_3(z_n)|\to\infty$ as $n\to \infty$, by Lemma~\ref{exp-dominating}.  
			
			\par
			 The remaining case is that  $L'$ is bounded  for every $j'\in\{0,1,2,\ldots,2m_2-1\}$ and every $\epsilon'>0$. Then, there is a $j^* \in \{0,1,2,\ldots, 2m_2 -1\}$ such that $L$ (as an unbounded subset of $\gamma$) is asymptotic to the fundamental ray  $R_{\alpha_{j^*} (Q_2)}$  (see  Remark~\ref{asymptotic-nont-connected}). Thus, we can construct an unbounded connected curve $\gamma^*$ using $\gamma$ that satisfies Property~(\ref{gamma-star}). More precisely,  $\gamma^*$ is as follows where $\epsilon^* \in (0, \frac{\pi}{2m_2})$:   if there is a $t_0 >0$ such that  $\alpha_{j^*} (Q_2) -\epsilon^* < Arg (\gamma(t)) < \alpha_{j^*} (Q_2) +\epsilon^* $ for all $t >t_0$ then $\gamma^* = \{\gamma(t): t>t_0\}$; otherwise there is $M>0$ such that the union of the two sets $  \left(\gamma  \cap  \{z: \alpha_{j^*} (Q_2) -\epsilon^*< Arg(z)< \alpha_{j^*} (Q_2) +\epsilon^*\}  \cap \{z: |z| > M\} \right)$ and $\{z: |z|=M\}$ is connected, which we take as $\gamma^*$.

			\par By Remark~\ref{unboundedpart-unboundedcurve} and Theorem~\ref{P_1e^{Q_1}+P_2}, the image of  $\gamma^*$  under $P_2 \exp(Q_2)+P_3$ is unbounded. In other words, there is a sequence  $z_n$ on  $\gamma^*$ such that   $\lim_{ n\to \infty} z_n  = \infty $ and $\lim_{n \to \infty} \left(P_2(z_n)\exp({Q_2(z_n)})  +  P_3(z_n)\right) =\infty$. Clearly $z_n \in L$ for all but possibly finitely many values of $n$.
			\par 
			Now we use the inequality 
			$$|F_3(z_n)|\geq  |P_2(z_n)\exp({Q_2(z_n)})  +  P_3(z_n)|-|P_1(z_n)\exp({Q_1(z_n)})|  $$
			to conclude that  $|F_3(z_n)|\to\infty$ as $n\to \infty$ as $\lim_{n \to \infty} P_1 (z_n) \exp(Q_1 (z_n)) =0$. 
			
			\item (\textit{Unbounded in an even-indexed closed sub-sector of $Q_1$})  If  $ \gamma $ is unbounded in an even-indexed closed sub-sector of $Q_1$,  then there is an even $j\in\{0,1,2,\ldots,2m_1-1\}$ and a positive $\epsilon$  such that $ \gamma \bigcap \overline{B^\epsilon_j(Q_1)}$ is unbounded. Take a sequence $\{z_n\}_{n\geq 1}$ on $\gamma\bigcap \overline{B^\epsilon_j(Q_1)}$  satisfying $\lim_{n\to \infty}z_n=\infty$. Applying  Lemma~\ref{Growth-exppoly}(1) to $Q_1$ and   Lemma~\ref{growth-poly} to $P_1$, we find constants $M_1,C_1>0$ such that $|P_1(z_n)\exp({Q_1(z_n)})|\geq M_1|z_n|^{d_1}\exp({C_1{|z_n|^{m_1}}})$. Applying Lemma~\ref{growth-poly} to $P_2, Q_2$ and $P_3$, we get $M_2,C_2$ and $M_3>0$ such that $|P_2(z_n)\exp({Q_2(z_n)})| \leq  M_2|z_n|^{d_2}\exp({C_2{|z_n|^{m_2}}})$ and $|P_3(z_n)| \leq M_3 |z_n|^{d_3}$ respectively. Therefore,  \begin{align*}
				|F_3(z_n)|&\geq |P_1(z_n)\exp({Q_1(z_n)})|-|P_2(z_n)\exp({Q_2(z_n)})|-|P_3(z_n)|\\&\geq  
				M_1|z_n|^{d_1}\exp({C_1{|z_n|^{m_1}}})-M_2|z_n|^{d_2}\exp({C_2{|z_n|^{m_2}}})-M_3|z_n|^{d_3}.
			\end{align*} 
			Since $m_1>m_2$, we have $|F_3(z_n)|\to\infty$ as $n\to \infty$, by Lemma~\ref{exp-dominating}.  
			
			\item (\textit{Asymptotic to a fundamental ray of $Q_1$}) If   $\gamma  $ is asymptotic to a fundamental ray of $Q_1$, then there is a $j \in \{0, 1, 2, \ldots, 2m_1 -1\}$ and a connected set $\gamma^*$ such that  
			every sequence on $\gamma^*$ going to $\infty$ has all but possibly finitely many of its terms lying on the original curve  $\gamma$ (see Property~(\ref{gamma-star})).
			
			\par Since the function $P_1 \exp({Q_1 })+P_3 $ has bov by Theorem \ref{P_1e^{Q_1}+P_2}, the image of $\gamma^*$ under $ P_1 \exp({Q_1 })+P_3$ is unbounded. In other words,  there exists a sequence $ z_n  \in  \gamma^*$ such that $\lim_{ n\to \infty}z_n=\infty$ and  $$|P_1(z_n)\exp({Q_1(z_n)})+P_3(z_n)|\to\infty~\mbox{as}~ n\to\infty.$$ As already observed, $\lim_{ n\to \infty}Arg(z_n)=\alpha_{j}(Q_1)$. Since none of the fundamental rays of $Q_1$ coincide  with those of $Q_2$, $z_n \in \overline{B^{\epsilon'}_{j'}(Q_2)}$ for some $j' \in \{0, 1, 2, \ldots ,2m_2 -1\}$ and positive $\epsilon$. Note that $j'$ is determined by $\alpha_{j}(Q_1)$ alone. There are two cases depending on the parity of $j'$.
			
			\par If $j'$ is odd, then  Lemma~\ref{Growth-exppoly}(2) applied to $Q_2$ gives $|P_2(z_n)\exp({Q_2(z_n)})|\to 0$ as $n\to\infty$.
			Using the inequality  $$|F_3(z_n)|\geq |P_1(z_n)\exp({Q_1(z_n)})+P_3(z_n)|-|P_2(z_n)\exp({Q_2(z_n)})|,$$ 
			we have $|F_3(z_n)|\to\infty$ as $n\to \infty$. 
			
			\par If $j'$ is even then, for all $z$ with sufficiently large modulus, $|\exp(Q_2 (z))| > \exp(C_2 |z|^{m_2})$ for some $C_2 >0$ by  Lemma~\ref{Growth-exppoly}(1).
			The unboundedness of  $F_3(\gamma^*)$ (therefore of $F_3 (\gamma)$) in this case is to be proved by contradiction. If  $F_3(\gamma^*)$ is bounded then 
			$$P_1(z)\exp({Q_1(z)})+P_2(z)\exp({Q_2(z)})+P_3(z)=\mathcal{O}(1),$$ 	
			as $z \to \infty$ along  $\gamma^*$.  Since, $P_2 (z) \neq 0$ for all $z$ with sufficiently large modulus,  we have
			\begin{equation}\label{order-equation}
				\frac{P_1(z)}{P_2(z)}\exp({Q_1(z)-Q_2(z)})+1=-\frac{P_3(z)}{P_2(z)}\exp({-Q_2(z)})+\frac{\mathcal{O}(1)}{P_2(z) \exp(Q_2(z))}.
			\end{equation}
			Applying Lemma~\ref{Growth-exppoly}(1) to $Q_2$, we have $\left|\frac{P_3(z)}{P_2(z)}\exp({-Q_2(z)})\right| \leq \left|\frac{P_3(z)}{P_2(z)} \right|  \exp(-C_2 |z|^{m_2})  $   and $\frac{\mathcal{O}(1)}{P_2(z) \exp(Q_2(z)} \to 0$ as $z\to\infty$ along $\gamma^*$. Consequently, the right hand side of the Equation (\ref{order-equation}) tends to $0$ and this gives that  $\frac{P_1(z)}{P_2(z)}\exp({Q_1(z)-Q_2(z)})\to-1 ~\mbox{as}~z\to\infty~\mbox{along}~\gamma^*.$  
			In other words, for every sequence $z_n \in \gamma^*$ with $\lim_{n \to \infty}z_n =\infty$, we have  $$\frac{P_1(z_n)}{P_2(z_n)}\exp({Q_1(z_n)-Q_2(z_n)})\to-1 ~\mbox{as}~n \to\infty.$$
			However, it follows from Lemma~\ref{asymptotic-to-ray} that there is a sequence $\{u_n\}_{n\geq 1}$ on $\gamma^*$  with $\lim_{n \to \infty}u_n =\infty$ such that  
			$-1$ is not a limit point of $\left\{\frac{P_1 (u_n)}{P_2(u_n)} \exp({Q_1(u_n)-Q_2(u_n)})\right\}_{n\geq 1}$.
			This is a contradiction, which completes the proof.  
	\end{enumerate}	\end{proof}	 
We now give the proof of Theorem~\ref{Generalization}.
	\begin{proof}[\bf{Proof of Theorem \ref{Generalization}}]
		For a natural number $l \geq 2$, let $$F_l(z)=P_1(z)\exp({Q_1(z)})+P_2(z)\exp({Q _2(z)})+\cdots+P_{l-1}(z)\exp({Q_{l-1}(z)})+P_l(z),$$ where $Q_i$ is a non-constant polynomial and $P_i$ is a non-zero polynomial (possibly constant) for each $i=1,2,3, \ldots, l-1$ and $P_l$ is a non-constant polynomial. 
		%	 It is also given that $\deg(Q_i)\neq \deg(Q_j)$  and none of the  fundamental rays of $Q_i$ coincide with those of $Q_j$ for any $i\neq j$.
		\par The proof is by the method of induction on $l$.  For $l=2,3$,  Theorem~\ref{Generalization} is in fact Theorems \ref{P_1e^{Q_1}+P_2} and  \ref{e^{Q_1 }+e^{Q_2 }+P_3}, respectively, which are already proved.
		Let $F_s$  has bov for all $s<l$, as the induction hypothesis.	
		
		\par 
		
		Let $\gamma:[0,\infty)\to\mathbb{C}$ be an unbounded curve. In view of Lemma~\ref{Bov iff}, it is enough to show that $F_l(\gamma)$ is unbounded. 
		Since $\deg(Q_i)\neq \deg(Q_j)$ for $i\neq j$, without loss of generality assume that $\deg(Q_i)>\deg(Q_{i+1})$ for all $i=1,2,3,\ldots,l-2$. Recall that $\deg(P_i)=d_i$ and $\deg(Q_i)=m_i$. There are four situations (not necessarily mutually exclusive) depending on the position of $\gamma$ with respect to the fundamental partition of $Q_1$.
		The first two are mutually exclusive and arise when $\gamma$ is unbounded in an odd-indexed closed sub-sector of $Q_1$. The third and fourth correspond to situations when $\gamma$ is unbounded in an even-indexed closed sub-sector of $Q_1$ and $\gamma$ is asymptotic to a fundamental ray of $Q_1$, respectively.
		
		\begin{enumerate}
			
			\item (\textit{Unbounded in an odd-indexed closed sub-sector of $Q_1$:I})	
			Let for each $i\in\{1,2,3,\ldots,l-1\}$, there exist some positive $\epsilon_i$ and odd $j_i\in\{0,1,2,\ldots,2m_{i}-1\}$ such that $\gamma\bigcap \left(\bigcap_{i=1}^{l-1} \overline{B^{\epsilon_{i}}_{j_i}(Q_i)}\right)$ is unbounded. Then for every  sequence $z_n \in \gamma\bigcap \left(\bigcap_{i=1}^{l-1} \overline{B^{\epsilon_{i}}_{j_i}(Q_i)}\right)$ with $\lim_{ n\to \infty}z_n=\infty$, from  Lemma~\ref{Growth-exppoly}(2) we have $$\lim_{n \to \infty} P_i (z_n) \exp(Q_i (z_n)) =0~ \mbox{for each }~i.$$ Using the inequality $$|F_l(z_n)|\geq |P_l(z_n)|-\sum_{i=1}^{l-1}|P_i (z_n) \exp(Q_i (z_n))|$$ we have $|F_l(z_n)|\to\infty$ as $n\to \infty$.
			
			\item (\textit{Unbounded in an odd-indexed closed sub-sector of $Q_1$:II})	Let there be an  $r\in\{1,2,3,
			\ldots,l-2\}$ such that for each $i\in\{1,2,3,\ldots,r\}$,  there exists some positive $\epsilon_i$ and some odd $j_i\in\{0,1,2,\ldots,2m_{i}-1\}$ such that $\Gamma:=\gamma\bigcap \left(\bigcap_{i=1}^{r} \overline{B^{\epsilon_{i}}_{j_i}(Q_i)}\right)$ is unbounded whereas $\Gamma \bigcap \overline{B^{\epsilon }_{j_{r+1}}(Q_{r+1})}$ is bounded for each odd $j_{r+1}\in\{0,1,2,\ldots,2m_{r+1}-1\}$ and for each positive $\epsilon$. Further, let each $\epsilon_i$ be chosen such that no boundary ray of any closed sub-sector $B_{j_i} ^{\epsilon_i}(Q_i), i =1,2,3,\ldots, r$ is a fundamental ray of $Q_{r+1}$. This choice is going to be useful in the case (b) below. For every sequence $\{z_n\}_{n\geq 1}$ on $\Gamma$ with $\lim_{n \to \infty} z_n =\infty$, it is observed from  Lemma~\ref{Growth-exppoly}(2) that $\lim_{n \to \infty} P_i (z_n) \exp(Q_i (z_n)) =0$ for each $i\in\{1,2,3,\ldots,r\}$. There are two different scenarios.
			\begin{enumerate}
				\item  The set  $\Gamma\cap\overline{B^{\epsilon}_{j_{r+1}}(Q_{r+1})}$ is unbounded for some even $j_{r+1}\in\{0,1,2,\ldots,2m_{r+1}-1\}$.  Consider a sequence $ z_n \in \Gamma$ with $\lim_{ n\to \infty}z_n=\infty$. Then,
				\begin{equation} \label{aux-1}
					\begin{split}
						|F_l(z_n)|\geq |P_{r+1}(z_n)\exp({Q_{r+1}(z_n)})|-\sum_{i=1}^{r}|P_i (z_n) \exp(Q_i (z_n))|-\\\sum_{i=r+2}^{l-1}|P_i (z_n) \exp(Q_i (z_n))|-|P_l(z_n)|.
					\end{split}
				\end{equation} Applying   Lemma~\ref{growth-poly} to $P_{r+1}$ and   Lemma~\ref{Growth-exppoly}(1) to $Q_{r+1}$, we get  $$|P_{r+1}(z_n)\exp({Q_{r+1}(z_n)})|\geq M_{r+1}|z_n|^{d_{r+1}}\exp({C_{r+1}{|z_n|^{m_{r+1}}}}), $$  for some constants $ M_{r+1},C_{r+1}>0$. 
				Similarly, applying Lemma~\ref{growth-poly} to $P_i, Q_i$ for each $i\in\{r+2,r+3,r+4,\ldots,l-1\}$, and to $P_{l}$, we get $$|P_i(z_n)\exp({Q_i(z_n)})| \leq  M_i|z_n|^{d_i}\exp({C_i{|z_n|^{m_i}}})~\mbox{ and}~ |P_{l}(z_n)| \leq M_{l} |z_n|^{d_{l}},$$ for some $M_i,M_{l}~\mbox{and}~C_i>0$. Putting these estimates in Inequality~(\ref{aux-1}), we have
				\begin{equation*}
					\begin{split}
						|F_l(z_n)|\geq M_{r+1}|z_n|^{d_{r+1}}\exp({C_{r+1}{|z_n|^{m_{r+1}}}})-\sum_{i=1}^{r}|P_i (z_n) \exp(Q_i (z_n))|-\\\sum_{i=r+2}^{l-1}M_i|z_n|^{d_i}\exp({C_i{|z_n|^{m_i}}})-M_{l} |z_n|^{d_{l}}.
					\end{split}
				\end{equation*} Since $\sum_{i=1}^{r}|P_i (z_n) \exp(Q_i (z_n))|\to 0$ as $n\to\infty$ by  Lemma~\ref{Growth-exppoly}(2), and $m_{r+1}>m_i$ for all $i\in\{r+2,r+3,r+4,\ldots,l-1\}$, it follows from Lemma~\ref{exp-dominating} that $|F_l(z_n)|\to\infty$ as $n\to \infty$.
				\item   
				The set	$\Gamma \bigcap \overline{B^{\epsilon}_{j_{r+1}}(Q_{r+1})}$ is bounded for each even $j_{r+1}\in\{0,1,2,\ldots,2m_{r+1}-1\}$ and for each  positive $\epsilon$. Along with the previous assumption, it means that the set $\Gamma \bigcap \overline{B^{\epsilon}_{j_{r+1}}(Q_{r+1})}$ is bounded for each $j_{r+1}$ and positive $\epsilon$. In other words, there is $j_{r+1}^* \in \{0,1,2,\ldots,2m_{r+1}-1\}$ such that $\Gamma$ is asymptotic to the fundamental ray   $R_{\alpha_{j_{r+1}^*}(Q_{r+1})}$. This fundamental ray is interior to  $\cap_{i=1}^{r}\overline{B^{\epsilon_{i}}_{j_i}(Q_i)}$.  By Remark~\ref{asymptotic-nont-connected}, there is an unbounded connected curve $\Gamma^*$ such that $\Gamma \cap \Gamma^*$ is unbounded. In particular, if $z_n \in \Gamma^*$ is a sequence with $\lim_{n \to \infty}z_n =\infty$ then all but possibly finitely many terms of the sequence are on $\Gamma^*$.

				By the induction hypothesis, $\sum_{i=r+1}^{l-1}P_i  \exp(Q_i )+P_l $  has bov. Therefore, the image of $\Gamma^*$ under this function is unbounded by Lemma~\ref{Bov iff}. Consequently, there exists a sequence $ z_n \in \Gamma^*$ with $\lim_{ n\to \infty}z_n=\infty$ such that
				$$ \sum_{i=r+1}^{l-1}P_i (z_n) \exp(Q_i (z_n))+P_l(z_n) \to \infty~\mbox{ as}~ n \to \infty.$$
				Since $z_n \in \Gamma$ for all but possibly finitely many values of $n$,   $$\sum_{i=1}^{r}|P_i (z_n) \exp(Q_i (z_n))| \to 0~\mbox{ as} ~n \to \infty,$$ by  Lemma~\ref{Growth-exppoly}(2). Now, using
				$$|F_l(z_n)|\geq \left|\sum_{i=r+1}^{l-1}P_i (z_n) \exp(Q_i (z_n))+P_l(z_n)\right|-\sum_{i=1}^{r}|P_i (z_n) \exp(Q_i (z_n))|,$$ it follows that $|F_l(z_n)|\to\infty$ as $n\to \infty$.
			\end{enumerate}

			\item (\textit{Unbounded in an even-indexed closed sub-sector of $Q_1$}) Let $\gamma \bigcap \overline{B^{\epsilon}_{j}(Q_1)}$ be unbounded for some even $j\in\{0,1,2,\ldots,2m_1-1\}$ and for some  positive $ \epsilon$. Consider a sequence $ z_n \in \gamma\bigcap \overline{B^{\epsilon}_{j}(Q_1)}$ with $\lim_{n\to \infty}z_n=\infty$. Applying  Lemma~\ref{Growth-exppoly}(1) to $Q_1$ and   Lemma~\ref{growth-poly} to $P_1$, we find constants $M_1,C_1>0$ such that $|P_1(z_n)\exp({Q_1(z_n)})|\geq M_1|z_n|^{d_1}\exp({C_1{|z_n|^{m_1}}})$. Similarly, for each $i=2,3,4,\ldots,l-1$, applying Lemma~\ref{growth-poly} to $P_i, Q_i$  and to $P_{l}$, we get $|P_i(z_n)\exp({Q_i(z_n)})| \leq  M_i|z_n|^{d_i}\exp({C_i{|z_n|^{m_i}}})$ and $|P_{l}(z_n)| \leq M_{l} |z_n|^{d_{l}}$ for some $M_i,M_{l}~\mbox{and}~C_i>0$. Therefore,  the inequality
			\begin{equation*}
				\begin{split}
					|F_l(z_n)|\geq |P_1(z_n)\exp({Q_1(z_n)})|-|P_2(z_n)\exp({Q_2(z_n)})|-\cdots\\-|P_{l-1}(z_n)\exp({Q_{l-1}(z_n)})|-|P_{l}(z_n)|
				\end{split}
			\end{equation*}
			can be written as
			\begin{equation*}
				\begin{split}
					|F_l(z_n)|\geq M_1|z_n|^{d_1}\exp({C_1{|z_n|^{m_1}}})-M_2|z_n|^{d_2}\exp({C_2{|z_n|^{m_2}}})-\cdots\\-M_{l-1}|z_n|^{d_{l-1}}\exp({C_{l-1}{|z_n|^{m_{l-1}}}})-M_{l}|z_n|^{d_{l}}.
				\end{split}
			\end{equation*}
			
			Since $m_1>m_r$ for each $r=2,3,4,\ldots,l-1$, we have $|F_l(z_n)|\to\infty$ as $n\to \infty$, by Lemma~\ref{exp-dominating}.
			\item (\textit{Asymptotic to a fundamental ray of $Q_1$})
			Let $\gamma$ be asymptotic to a fundamental ray of $Q_1$. Then, there exists   $j^* \in \{0,1,2,\ldots,2m_1-1\}$ and an unbounded curve $\gamma^*$ such that $\gamma \cap \gamma^*$ is unbounded, and  every sequence $z_n \in \gamma^*~$ with   $\lim_{n \to \infty}  z_{n }   =\infty$   satisfies $\lim_{n \to \infty} Arg(z_{n }) = \alpha_{j^* }(Q_1)$   (see Property~(\ref{gamma-star})).
			\par By the hypothesis of this theorem (not the induction hypothesis), none of the fundamental rays of $Q_2$ coincide with those of $Q_1$. This gives that, there exists   $j_2 \in \{0,1,2,\ldots, 2m_2 -1\}$ and some positive $\epsilon_2$ such that $\gamma^* \cap \overline{B_{j_2}^{\epsilon_2} (Q_2)}$ is unbounded. 

			\par By the induction hypothesis, the function $F_l - P_2 \exp(Q_2)= 	  P_1\exp({Q_1})+P_3\exp({Q _3})+P_4\exp({Q _4})+\cdots+P_{l-1}\exp({Q_{l-1}})+P_{l} $  has bov. The image of the unbounded curve $ \gamma^{*}$ under $F_l -P_2 \exp(Q_2)$ is unbounded by Lemma~\ref{Bov iff}. In other words,
			\begin{equation}\label{aux-2}
				\parbox{0.9\textwidth}{
					there exists a sequence $z_n\in \gamma ^{*}$ such that 
					$\lim_{ n\to \infty}z_n=\infty$, $\lim_{ n\to \infty}Arg(z_n)=\alpha_{j^*}(Q_1)$, and $|F_l(z_n) -P_2 (z_n)\exp(Q_2 (z_n))|\to\infty$ as $n\to\infty$.
				}
			\end{equation}

			There are now two cases depending on the parity of $j_2$.
			
			\begin{enumerate}
				\item  If $j_2$ is odd, then  Lemma~\ref{Growth-exppoly}(2) applied to $Q_2$ gives $|P_2(z_n)\exp({Q_2(z_n)})|\to 0$ as $n\to\infty$.
				Since $$|F_l(z_n)|\geq |F_l(z_n)  -P_2 (z_n)\exp(Q_2 (z_n)) |-|P_2(z_n)\exp({Q_2(z_n)})|,$$
				it follows from (\ref{aux-2}) that $|F_l(z_n)|\to\infty$ as $n\to \infty$.
				
				\item If $j_2$ is even then there exists $C_2 >0$ such that $|\exp(Q_2 (z))| > \exp(C_2 |z|^{m_2})$ for all $z \in \gamma^{*}$  with sufficiently large modulus by Lemma~\ref{Growth-exppoly}(1). We now analyze $\frac{F_l}{P_2 \exp(Q_2)}$	on $\gamma^{*}$.	  Since  $P_2 (z) \neq 0$ for all $z$ with sufficiently large modulus,  we have
				\begin{equation}\label{order-equation-2}
					\begin{split}
						\frac{F_l (z)}{P_2 (z)\exp(Q_2(z))}=	\frac{P_1(z)}{P_2(z)}\exp({Q_1(z)-Q_2(z)})+1+\frac{P_3(z)}{P_2(z)}\exp({Q_3(z)-Q_2(z)})+ \\
						\cdots    
						+\frac{P_{l-1}(z)}{P_2(z)}\exp({Q_{l-1}(z)-Q_2(z)})	
						+\frac{P_{l}(z)}{P_2(z)}\exp({-Q_2(z)}).
					\end{split}
				\end{equation} 
				By Lemma~\ref{growth-poly}, there exist   $C_i>0$ for $i=3,4,5,\ldots,l-1$ such that
				$\left| \exp({Q_i(z)})\right| \leq   \exp(C_i|z|^{m_i})$. Therefore,
				$$\left| \exp({Q_i(z)-Q_2(z)})\right| \leq   \exp(C_i|z|^{m_i}-C_2 |z|^{m_2}),~\mbox{where each}~m_i<m_2.$$
				Clearly, each term in the right-hand side of Equation~(\ref{order-equation-2}) except the first two goes to $0$ as $z \to \infty$.	
				\par 
				Now, if $F_l(\gamma^{*})$ is bounded then the left-hand side of Equation~(\ref{order-equation-2})	goes to $0$ as $z \to \infty$ along $\gamma^{*}.$ Consequently,   $$\frac{P_1(z)}{P_2(z)}\exp({Q_1(z)-Q_2(z)})\to-1 ~\mbox{as}~z\to\infty~\mbox{along}~\gamma^{*}.$$
				However, this is a contradiction to Lemma \ref{asymptotic-to-ray}.
			\end{enumerate}	
		\end{enumerate}
		\par This completes the proof.
	\end{proof} 
	
	\section{No Baker wandering domain}\label{No BWD}
	In this section, we prove a theorem from which it follows that the entire functions defined in Theorems \ref{P_1e^{Q_1}+P_2}, \ref{e^{Q_1 }+e^{Q_2 }+P_3} and \ref{Generalization} do not have any Baker wandering domain.  For this, a result proved by Bergweiler et al. \cite{BergRippStall 2013} is required.
	\begin{lemma}\label{Lem7}
		Let $f$ and $h$ be transcendental entire functions and suppose that there exists $\alpha\in(0,1)$ such that $$M(r,h-f)\leq {M(r,f)}^\alpha.$$ If $f$ has a Baker wandering domain then $h$ also has a Baker wandering domain.	
	\end{lemma}

	\begin{theorem}\label{No-BWD-result}
		For a natural number $l\geq 2$, if	$P_1,P_2,\ldots,P_l$ are non-zero polynomials, possibly constant and $Q_1, Q_2,\ldots, Q_{l-1}$ are non-constant polynomials such that $\deg(Q_i)\neq \deg(Q_j)$ for $i\neq j$ then the entire function $P_1\exp(Q_1)+P_2\exp(Q _2)+\cdots+P_{l-1}\exp(Q_{l-1})+P_l$ does not have any Baker wandering domain.
		
	\end{theorem}
	\begin{proof} Let $\deg(P_i)=d_i$, $\deg(Q_i)=m_i$ for each $i\in\{1,2,3,\ldots,l-1\}$ and $\deg(P_l)=d_l$.
		
		If $l=2$, let $f(z)=P_1(z)e^{Q_1(z)}+P_2(z)$. Let $h(z)=P_1(z)e^{Q_1(z)}$. It can be  seen that $$M(r,h-f)=\max_{|z|=r}|P_2(z)|\leq {r}^{d_2+1} ~\mbox{for sufficiently large}~ r.$$ Therefore, ${M(r,h-f)}^2\leq r^{2({d_2+1})}$. It is well known that for any $C>0$ and sufficiently large $r$, any transcendental function $\phi$ satisfies the relation $M(r,\phi)>r^C$ (see page 11, \cite{asymptoticvalue-book}). Choosing $C=3({d_2+1})$, we have  ${M(r,h-f)}^2\leq r^{2({d_2+1})}< r^C< M(r,f)$. Thus, $M(r,h-f)< {M(r,f)}^{\frac{1}{2}}$. We are going to show this inequality for $l \geq 3.$
		\par 
		
		If $l\geq 3$, let $f(z)=P_1\exp(Q_1)+P_2\exp(Q _2)+\cdots+P_{l-1}\exp(Q_{l-1})+P_l$, where $P_1,P_2,\ldots,P_l$. Take $h(z)=P_1(z)e^{Q_1(z)}$. Since the degrees of $Q_1,Q_2,\ldots,Q_{l-1}$ are pairwise distinct, without loss of generality, we assume that $m_1 >m_i $ for every $i$. Then, for sufficiently large $r$, we have
		\begin{align*} M(r,h-f)&=\max_{|z|=r}\left|\sum_{i=2}^{l-1}P_i(z)\exp(Q_i(z))+P_l(z)\right|\ \\& \leq\max_{|z|=r}\left(\left(\sum_{i=2}^{l-1}|P_i(z)|\exp(|Q_i(z)|) \right)+|P_l(z)|\right). \end{align*}
		Applying Lemma~\ref{growth-poly} to each $P_i$ and $Q_i$, it is observed that $$M(r,h-f) \leq   \max_{|z|=r}\left(\left(\sum_{i=2}^{l-1}M_i |z|^{d_i}\exp(\widetilde{M_i}|z|^{m_i}) \right)+M_l |z|^{d_l}\right),$$
		for some positive constant $M_i, \widetilde{M_i}$.
		Since  $M_i |z|^{d_i} < \exp(\widetilde{M_i}|z|^{m_i})$ for $2 \leq i \leq l-1$ and for all sufficiently large $r$, denoting $ \max\{\widetilde{M_i}: 2 \leq i \leq l-1\}$ and $ \max\{m_i: 2 \leq i \leq l-1\}$ by $C$ and $m$ respectively, we have $M(r,h-f) \leq   (l-2) \exp(2Cr^{m}) + M_l r^{d_l}$. Since $M_l r^{d_l} < \exp(2 Cr^m)$ for all sufficiently large $r$, we have
		
		\begin{equation}\label{upperbound}
			M(r,h-f)   < (l-1) \exp(2Cr^{m}).\end{equation}

		\par  Let  $a_{m_1}$ be the leading coefficient of $Q_1$ and $z_r$ be a  point at which $\Re(a_{m_1}z^{m_1})$ attains its maximum on the circle $\{z: |z|=r\}$. Then $ \Re(a_{m_1}z_{r}^{m_1})=|a_{m_1}|r^{m_1}.$
		
		Observe that $\max_{|z|=r} \Re(Q_1(z))  \geq \Re(Q_1(z_r)) \geq |a_{m_1}|r^{m_1} - |a_{m_1 -1}|r^{m_1-1}-  \cdots- |a_{1}|r -|a_0| $. Applying Lemma~\ref{growth-poly} to the right-hand side we see that  $\max_{|z|=r} \Re(Q_1(z))  \geq C' r^{m_1}$ for some $C'>0$   and for all sufficiently large $r$. Also, applying this lemma to $P_1$, it is seen that there is $M_1 >0$ satisfying $|P_1 (z)| \geq M_1 |z|^{d_1}$ for all $z$ with sufficiently large modulus. Therefore $|P_1(z_r)e^{Q_1(z_r)}|\geq M_1 r^{d_1} \exp(C' r^{m_1}).$ Thus, for sufficiently large $r$, $$\left|f(z_r)\right|\geq |P_1(z_r)e^{Q_1(z_r)}|-\left|h(z_r)-f(z_r)\right| \geq M_1 r^{d_1} \exp(C' r^{m_1}) - (l-1) \exp(Cr^m).$$

		The right-hand side can be seen to be larger than $  \exp(A r^{m_1})$ for some $A>0$ (as $m_1 >m$) and consequently $M(r,f) \geq |f(z_r)| \geq\exp(Ar^{m_1})$ for all sufficiently large $r$. Thus, $M(r,f)^{\frac{1}{2}}\geq\exp\left(\frac{A}{2}r^{m_1}\right)$. Since $m_1>m $, we have $\log (l-1)  Cr^{m} < \frac{A}{2}r^{m_1}$ for all sufficiently large $r$. By taking exponential on both the sides, $ (l-1)\exp(C r^{m}) < \exp(\frac{A}{2}r^{m_1}) \leq M(r,f)^{\frac{1}{2}}$. In other words,  $M(r,h-f)<M(r,f)^{\frac{1}{2}}$ for all sufficiently large $r$.
		
		\par
		From Lemma \ref{Lem7}, it is evident that if $f$ has a Baker wandering domain, then so does $h$. In other words, the absence of Baker wandering domains for $h$ implies the absence of Baker wandering domains for $f$. However, $h$ has at most finitely many zeros, and it implies that $h$ does not have any Baker wandering domain by Remark \ref{Rem2.2}. Consequently, $f$ does not have any Baker wandering domain.	
	\end{proof}
	\begin{Remark}
		\begin{enumerate}
			\item   In \cite{Zheng 2006} Zheng proved that, for rational functions $R_i$ and polynomials $Q_i$,  the meromorphic function $\sum_{i=1}^{m}R_i(z)e^{Q_i(z)}$ does not have any Baker wandering domain. Theorem~\ref{No-BWD-result} follows from this result, Zheng's proof is slightly more involved than ours.
			\item By Theorem 3.1 of \cite{Baker 1984}, a Fatou component of an entire function is a Baker wandering domain if and only if it is multiply connected. Hence, every Fatou component of the functions in the above theorem is simply connected.
		\end{enumerate}
		
	\end{Remark}
	
	\section{Other examples and concluding remarks}\label{Examples}
	In Theorem \ref{e^{Q_1 }+e^{Q_2 }+P_3}, the hypotheses are that $\deg(Q_1)\neq \deg(Q_2)$ and none of the fundamental rays of $Q_1$ coincides with those of  $Q_2$. But there are functions with bov,  for which none of these assumptions holds.
	
	\begin{example}\label{Exm}
		For every non-constant polynomials $P$ and $Q$, the functions $\exp({Q})+\exp({-Q})+P$ and $\exp({Q})-\exp({-Q})+P$ have bov.
	\end{example}
	\begin{proof}
		We show that $f(z) =\exp({Q})+\exp({-Q})+P$  has bov. The proof for $\exp({Q})-\exp({-Q})+P$ is analogous.
		\par
  Let $\gamma:[0,\infty)\to \mathbb{C}$ be an unbounded curve in the complex plane. It is enough to show that $f(\gamma)$ is unbounded by Lemma~\ref{Bov iff}. This is to be done by finding a sequence $\{z_n\}_{n\geq 1}$ on $\gamma$ with $\lim_{n \to \infty} z_n =\infty$ such that $\lim_{n\to \infty} f(z_n) =\infty$.
  
	\par If there is a sequence $\{z_n\}_{n\geq 1}$ on $\gamma$ with $\lim_{n \to \infty}z_n =\infty$ such that $\{\Re(Q(z_n))\}_{n \geq 1}$ is bounded then $|\exp(Q(z_n))+\exp(-Q(z_n))| \leq \exp(\Re(Q(z_n)))+\exp(-\Re(Q(z_n))) \leq K$ for some $K>0$.
		Therefore, $|f(z_n)| \geq |P(z_n)|-|\exp({Q(z_n)})+\exp({-Q(z_n)})| \geq |P(z_n)|-K$, which gives that $\lim_{n\to \infty} f(z_n) =\infty$.
		\par Let 
		\begin{equation}
			\label{for-every-sequence}
			\mbox{for every sequence}~u_n \in \gamma~\mbox{with}~\lim_{n\to \infty} u_n =\infty,~\mbox{the sequence}~\{\Re(Q(u_n))\}_{n \geq 1}\mbox{is unbounded}.
		\end{equation}  
		Take a point $w_0 \in Q(\gamma)$, let $\Re(w_0)=M$ and consider the set $\{\Im(w): w \in Q(\gamma)~\mbox{and}~\Re(w)=M\}$. If this set is unbounded then there is a sequence $\{u_n'\}_{n\geq 1}$ in $\gamma$ with $\Re(Q(u_n'))=M$ and $\lim_{n \to \infty} \Im(Q(u_n'))$ is either $+\infty$ or $-\infty$. Since $Q$ is a non-constant polynomial, we must have $\lim_{n \to \infty} u_n' =\infty$. However, this is not possible by (\ref{for-every-sequence}). Thus, there is an $M' >0$ such that for every $w \in Q(\gamma)$ with $\Re(w)=M$, we have $|\Im(w)|<M'$. Consequently, at least one of the sets $\{w: w \in Q(\gamma) ~\mbox{and}~\Re(w) > \Re(w_0)\} $ or $ \{w: w \in Q(\gamma) ~\mbox{and}~\Re(w) < \Re(w_0)\} $ is unbounded. Without loss of generality, assume the first situation. 
		\par Consider the set $$\Gamma:= \{w: w \in Q(\gamma) ~\mbox{and}~\Re(w) > \Re(w_0)\} \cup \{w: \Re(w)=\Re(w_0)  ~\mbox{and}~|\Im(w)|<M' \}.$$
		This is the union of the part of $Q(\gamma)$ lying right to the vertical line passing through $w_0$ and a bounded connected subset of this vertical line, which contains all the points of intersection of $Q(\gamma)$ and the line. The set $\Gamma$ is clearly  connected and unbounded. Let $\gamma^*$ be a component of $Q^{-1}(\Gamma)$ intersecting $\gamma$. Then $\gamma^*$ is also connected and unbounded. Since the function $\exp(Q)+P$ has bov, by Theorem~\ref{P_1e^{Q_1}+P_2}, the image of $\gamma^*$ under this function is unbounded by Lemma~\ref{Bov iff}. In other words, there is a sequence $z_n \in \gamma^*$ with $\lim_{n \to \infty}z_n =\infty$ such that $\lim_{n \to \infty}\exp(Q(z_n))+P(z_n)=\infty$.	Note that the set $\{ z \in \gamma^*: \Re(Q(z)) =M~\mbox{and}~|\Im(Q(z))|<M'\}$ is bounded and therefore, all but possibly finitely many terms of $\{z_n\}_{n \geq 1}$ are on the original curve $\gamma$. Thus, $\lim_{n \to \infty}\Re(Q(z_n))=+\infty $ and consequently,
		$|f(z_n)| \geq |P(z_n)+\exp({Q(z_n)})|- |\exp({-Q(z_n)})|.$ This gives that $\lim_{n \to \infty}|f(z_n)| = \infty$.		
	\end{proof}
	\begin{Remark}
		Taking $Q(z)=z$ in the above example, it is seen that for every non-constant polynomial $P$, the functions $\sinh(z)+P(z)$ and $\cosh(z)+P(z)$ have bov.
	\end{Remark}
	
	We conclude by posing the following questions.
	\begin{question}
		\begin{enumerate}
			\item  In Example~\ref{Exm}, the conclusion seems to hold if $-Q$ is replaced by a polynomial with the same degree as that of $Q$ and whose leading coefficient is the negative of that of $Q$. This, however, remains to be verified. More generally, one may consider any polynomial $\widetilde{Q}$ in place of $-Q$ such that the set of fundamental rays of $\widetilde{Q}$ and $-Q$ coincide.
			\item It would be interesting to see up to what extent the hypotheses on the degrees and leading coefficients of $Q_1, Q_2$ can be relaxed in Theorem~\ref{e^{Q_1 }+e^{Q_2 }+P_3}. For example, keeping the degrees the same, the leading coefficients can be chosen suitably to ensure that the fundamental rays do not coincide.  However, additional care is needed regarding the inequalities involved in the proof.
			\item Each of the functions considered in Theorems ~\ref{P_1e^{Q_1}+P_2}, \ref{e^{Q_1 }+e^{Q_2 }+P_3} and \ref{Generalization} has an unbounded set of critical values, while these do not have any Baker wandering domain. The following questions are natural. 
			\begin{enumerate}
				\item  There are examples having Baker domain such as Fatou's function $1+z +e^{-z}$ \cite{Fatou 1920}. In \cite{DasNayak 2024}, it is shown that the function $  \lambda +z+ \lambda e^z$ does not have any Baker domain for $ 0< \lambda < 2$. Characterization of functions considered in Theorems ~\ref{P_1e^{Q_1}+P_2}, \ref{e^{Q_1 }+e^{Q_2 }+P_3} and \ref{Generalization} that has a Baker domain, remains to be done.
				\item In view of \cite{Simplyconnected-Rempeetal}, determining the maximum number of completely invariant domains for these functions is also an interesting problem.
				\item In general, these functions can be considered as well-behaving test cases for every result that are true for functions in the Eremenko-Lyubich class.
			\end{enumerate}
		\end{enumerate}
	\end{question}
	% 	\section{Disclosure statement}
	% 	The authors report that there is no competing interest to declare.
 \section{Funding}
	The first author is supported by the University Grants Commission, Govt. of India, through a Senior Research Fellowship.
	% 	\section{Data Availability statement} Data sharing not applicable to this article as no datasets were generated or analyzed during the current study.
	%	
	 
\end{document}